\documentclass[11pt]{amsart}
\usepackage{amsmath,amssymb,amsthm}
\usepackage[margin=1.15in]{geometry}
\usepackage{xcolor}
\usepackage[hidelinks]{hyperref}

\theoremstyle{plain}
\newtheorem{theorem}{Theorem}[section]
\newtheorem{proposition}[theorem]{Proposition}
\newtheorem{prop}[theorem]{Proposition}
\newtheorem{lemma}[theorem]{Lemma}
\newtheorem{corollary}[theorem]{Corollary}
\newtheorem{conjecture}[theorem]{Conjecture}
\theoremstyle{definition}
\newtheorem{definition}[theorem]{Definition}
\newtheorem{remark}[theorem]{Remark}

\newcommand{\one}{\mathbf 1}
\newcommand{\Zp}{\Z_{+}}
\newcommand{\h}{\mathfrak h}
\newcommand{\gr}{\Upsilon}

\newcommand{\ii}{\sqrt{-1}}

\DeclareMathOperator{\ad}{ad}

\newcommand\ov{\overline}

\renewcommand\({\left(}
\renewcommand\){\right)}

\newcommand\be{\beta}

\newcommand\g{\mathfrak g}

\renewcommand\l{\lambda}

\renewcommand\d{\delta}

\renewcommand\a{\alpha}

\newcommand{\Z}{\mathbb Z}

\newcommand\Q{\mathbb Q}

\newcommand\la{\lambda}

\newcommand\e{\epsilon}
\newcommand\C{\mathbb C}
\newcommand\R{\mathbb R}
\newcommand\si{\sigma}

\newcommand{\rank}{{\rm rank}}

\newcommand{\sdim}{\text{\rm sdim}}
\newcommand{\Gr}{\operatorname{Gr}}

\usepackage{amsmath,amssymb}
\usepackage{tikz}
\usetikzlibrary{decorations.markings}
\tikzset{
  dyn/.style={baseline=-0.6ex, x=0.95cm, line width=0.8pt},
  root/.style={circle, draw, line width=0.8pt, minimum size=8pt, inner sep=0pt},
  E/.style={root, label={[font=\scriptsize, inner sep=1.5pt]above:$0$}},
  O/.style={root, label={[font=\scriptsize, inner sep=1.5pt]above:$\tfrac12$},
    path picture={\draw[thin]
      (path picture bounding box.north west) -- (path picture bounding box.south east)
      (path picture bounding box.south west) -- (path picture bounding box.north east);}},
  dots/.style={inner sep=1pt},
  arrowmark/.style={postaction={decorate, decoration={markings,
    mark=at position .5 with {\draw[line width=0.8pt] (-3pt,5pt) -- (2.5pt,0) -- (-3pt,-5pt);}}}},
}

\tikzset{Eb/.style={root, label={[font=\scriptsize, inner sep=1.5pt]below:$0$}}}
\newcommand{\bea}{\begin{eqnarray}}
\newcommand{\eea}{\end{eqnarray}}

\begin{document}

\title[Unitary VOAS]{On unitary vertex operator algebras}

\author[Kac, M\"oseneder, Papi]{Victor G.~Kac}
\author[]{Pierluigi M\"oseneder Frajria}
\author[]{Paolo  Papi}

\begin{abstract}
We prove that a conformal vertex algebra with a non-zero invariant Hermitian form is
unitary as soon as it admits a non-zero positive energy module with a positive
definite invariant Hermitian form. Using free field realizations, we deduce that
the principal simple $W$-algebra $W_k(\g)$ of a simple Lie algebra $\g$ is unitary
for central charge $c\ge \rank(\g)+24\bigl((\rho|\rho^\vee)+\|\rho\|\,\|\rho^\vee\|\bigr)$, and,
if $\g$ is simply laced, for all $c\ge\rank(\g)$. Among basic Lie superalgebras
 which are not Lie algebras, the free field argument applies to $osp(1|2n)$ and
$osp(2|2n)$. It shows that $W_k(osp(1|2n))$ is unitary for $c\ge n+\frac12$ and
gives an explicit bound for $osp(2|2n)$.  Next,
we prove that if $\mathfrak s$ is an even $sl_2$-subalgebra of a simple Lie
algebra $\g$ whose centralizer is not abelian, then the simple
$W$-algebra $W_k(\g,\mathfrak s)$ is unitary for at most finitely many $k$. We
illustrate this on the $W$-algebras of $sl_{qm}$ attached to the nilpotents of
Jordan type $q^m$: unitarity is possible only at two levels, and for $q=2$ it
does occur at both. For instance, the $W$-algebra of $sl_4$ attached to the
partition $2+2$ is unitary only for $k=-2$ and $k=-\frac32$. This fails for Lie superalgebras, as minimal $W$-algebras show, but still holds under stronger assumptions. Finally, we prove that a simple $W$-algebra $W_k(\g,\mathfrak s)$ can be unitary for  only  finitely many   $k$ if $\g$ is a simple Lie algebra and $\mathfrak s$ is not a distinguished $sl_2$ subalgebra.
\end{abstract}

\maketitle

\section{Introduction}
The most important result of this paper is the following general theorem on unitarity of conformal vertex algebras (= vertex operator algebras), which was posed by us as Conjecture 3 in \cite{KMP1}.
\begin{theorem}\label{11} If a 
  conformal vertex algebra $V$ endowed with an invariant Hermitian form admits a non-zero unitary positive energy module, then $V$ is unitary.
\end{theorem}
Recall that to any $sl_2$-subalgebra $\mathfrak s= span\{e,x,f\}$ with $[e,f]=x, [x,e]=e, [x, f]=-f$, of a simple Lie algebra, or, more generally,  basic Lie superalgebra $\g$, and $k\in\C$, different from $k_{crit}=- h^\vee$, one associates the freely generated universal $W$-algebra $W^k(\g,\mathfrak s)$ and its (unique) simple quotient  $W_k(\g,\mathfrak s)$, which are conformal vertex algebras. If $\mathfrak s$ is  a principal $sl_2$-subalgebra, these conformal vertex algebras are called {\it principal}. Using Theorem \ref{11} we prove
\begin{theorem}\label{12} If $\g$ is a simple Lie algebra, then the principal simple $W$-algebra $W_k(\g)$ is unitary for the following values of the central charge
$$c_{\g,k}\ge m_0:=\rank(\g)+24\big((\rho|\rho^\vee)+\|\rho\|\cdot\|\rho^\vee\|\big).$$
If, moreover, $\g$ is simply laced, then  $W_k(\g)$ is  unitary for
$c\ge\rank(\g)$.
\end{theorem}
The two bounds are Theorems \ref{4.1} and   \ref{4.2}. For $\g=sl_3$
the second statement, i.e.\ the unitarity of the $W_3$-algebra for $c\ge2$, was
proved by Carpi, Tanimoto and Weiner \cite{CTW}. \par We  also analyze the super case. The outcome is as follows
  \begin{theorem}\label{14} Let $\g$ be a basic Lie superalgebra which is  not a Lie algebra and let $W_k(\g)$ be its principal simple $W$-algebra.  The free field argument proving Theorem \ref{12} applies to $\g=osp(1|2n)$ and $\g=osp(2|2n)$, $n\ge1$, and, for involutions of diagram type (Remark \ref{s:remscreen}), only to these. It gives:
$W_k(osp(1|2n))$ is unitary for  $c_{\g,k}\ge n+\frac12$; $W_k(osp(2|2n))$ is unitary for $c_{\g,k}\ge m_0+1$,  where $m_0$ is the
expression of Theorem \ref{12} computed for $\g=osp(2|2n)$, with $\rho^\vee=x$ and $\rho=\rho_{\ov 0}-\rho_{\ov 1}$: explicitly, for $c_{\g,k}\ge n+2+2n(n-1)(4n+1)+2\sqrt2\,n(2n-1)\sqrt{(n-1)(2n+1)}$.
\end{theorem}
The two bounds are Theorems \ref{s:thm5}  and   \ref{s:thm7}, respectively. 
The proof of the second part of Theorem \ref{12} uses the Feigin-Frenkel duality for simply laced simple Lie algebras \cite{FF}. The proof of the first part of Theorem \ref{14}  uses  a variant of FF-duality:
$W_k(osp(1|2n))\cong W_{k'}(osp(1|2n))$ with $(k+h^\vee)(k'+h^\vee)=\frac{1}{4}$ \cite{G}, \cite{CL}, \cite{CLNS}. Both dualities
are available in the literature for generic levels, while we need them at
non-real levels, which are not known to be generic. Therefore we prove in
Proposition \ref{prop:FFall}, using the injectivity of the Miura map \cite{N},
that they hold at all non-critical levels. \par
The obstruction to an extension of Theorem \ref{12} to all basic Lie superalgebras  is 
twofold: the invariant Hermitian form must be positive definite on a real
form of a Cartan subalgebra $\h$, fixed, up to a phase, by a symmetry of the
principal Dynkin diagram (Proposition \ref{s:prop4}), and, in order to go below
the bound given by real levels, $\rho$ and $\rho^\vee$ must be proportional
(Propositions \ref{s:prop2}, \ref{s:prop3}).  

Our next result is the following 
\begin{theorem}\label{31} If   $\g$ is a simple Lie algebra, $\mathfrak s$ is an even $sl_2$-subalgebra such that the centralizer of $\mathfrak s$ is not abelian, then the simple $W$-algebra $W_k(\g,\mathfrak s)$ can be unitary for at most a finite number of $k$.
\end{theorem}
In fact,   in Theorem \ref{prop:gen} we prove an enhanced version of Theorem \ref{31} for basic Lie superalgebras, and in Proposition \ref{prop:blocks} we provide an explicit criterion to verify the assumptions of Theorem \ref{prop:gen}
in the case of  basic  Lie superalgebras, which are not Lie algebras.
\par
We illustrate Theorem \ref{31} on the nilpotent elements $f_{q^m}$ of
$sl_{qm}$ of Jordan type $q^m$. 
We also show that, if $\g^\natural$ is semisimple, $W_k(\g,\mathfrak s)$ is unitary and
its central charge equals that of its affine subalgebra, then
$W_k(\g,\mathfrak s)$ coincides with its affine subalgebra (Proposition
\ref{prop:conformal}). The proof follows \cite{AMP23}; for $\g$ of
exceptional type all the relevant levels, except one ($E_7$, $D_4(a_1)$, $k=-12$),   have already been detected as collapsing
in \cite{AACLMMP}.
 In Section \ref{sec:sl4} we collect, in Table \ref{tab:sl4}, what is known on the unitarity of the $W$-algebras $W_k(sl_4,\mathfrak s)$ for all $\mathfrak s$; this includes the classification of the unitary levels for the subregular nilpotent element (Proposition \ref{prop:31}). \par
 Finally, in Section \ref{sec:nd} we prove that, for a simple Lie algebra $\g$, the conclusion of
Theorem \ref{31} holds  whenever
 the centralizer of $\mathfrak s$ is non-zero (Theorem \ref{nd:main}).
Since principal $W$-algebras are unitary for infinitely many $k$ (Theorem
\ref{12}), this leads us to conjecture that, for a simple Lie algebra $\g$, $W_k(\g,\mathfrak
s)$ is unitary for infinitely many $k$ if and only if $\mathfrak s$ is principal (Conjecture
\ref{nd:conj}). By Theorem \ref{nd:main}, the conjecture is open only for the distinguished
non-principal $sl_2$-subalgebras, which exist in types $B_n$ $(n\ge4)$, $C_n$ $(n\ge3)$, $D_n$
$(n\ge4)$, $E_6$, $E_7$, $E_8$, $F_4$, $G_2$; in particular it holds for $\g$ of type $A_n$,
$B_2$ and $B_3$. 
\par
The paper is divided into eight sections:   the first two are introductory; Section 3 deals with Theorem \ref{11}, Section 4 (resp. 5)  deals with  Theorem \ref{12} (resp. Theorem \ref{14}). Section 6 is devoted to 
Theorem \ref{31}; Section 7 deals with the $W$-algebras of $sl_4$, and Section 8
with the non-distinguished nilpotent elements (Theorem \ref{nd:main}) and with Conjecture
\ref{nd:conj}.\par
Throughout the paper the base field is $\C$.

\section{Notation and background}\label{sec:back}
\subsection{Conformal vertex algebras}
A \emph{conformal vertex algebra} is a vertex (super)algebra $V$ with a Virasoro vector
$L$ such that $L_{-1}=T$, $L_0$ is diagonalizable with the eigenspace decomposition
$V=\bigoplus_{n\in\frac12\Zp}V_n$, where $\dim V_n<\infty$ and $V_0=\C\one$
\cite[Def.~2.1]{KMP}.  We write $\Delta_a$ for the eigenvalue=conformal weight and $p(a)\in\{0,1\}$
for the parity of a $L_0$-eigenvector $a\in V$, and $p(a,b)=(-1)^{p(a)p(b)}$.  Since
$V_0=\C\one$, every $v\in V$ has a well defined \emph{expectation value} $\langle
v\rangle$, defined by $P_{V_0}(v)=\langle v\rangle\one$.

Set $\Gamma=\frac12\Z$. A $\Gamma/\Z$-\emph{grading} $\gr$ of $V$ is a
decomposition $V=V^{[0]}\oplus V^{[1/2]}$ into subsuperspaces, which is a vertex
algebra grading and is preserved by $L_0$ \cite[Def.~2.3]{KMP}; in particular
$\one\in V^{[0]}$. For the notion of a $\gr$-\emph{twisted} $V$-module $M$ we
refer to \cite[Def.~2.4]{KMP}; when $\gr$ is the trivial grading $V=V^{[0]}$ this
is the usual notion of a $V$-module. 

For $a\in V_{\Delta_a}\cap V^{[\gamma]}$ we write:
\[
Y(a,z)=\sum_{n\in\mathbb Z}a_{(n)}z^{-n-1}=\sum_{n\in -\Delta_a+\mathbb Z}a_n z^{-n-\Delta_a},
\qquad a_n=a_{(n+\Delta_a-1)},
\]
so that $a_nV_j\subseteq V_{j-n}$ and $a=a_{-\Delta_a}\mathbf 1$. On a $\Upsilon$-twisted
$V$-module $M$ we use the same conventions:
\[
Y^M(a,z)=\sum_{n\in\gamma+\mathbb Z}a^M_{(n)}z^{-n-1}=\sum_{n\in\gamma-\Delta_a+\mathbb Z}a^M_n z^{-n-\Delta_a},
\qquad a^M_n=a^M_{(n+\Delta_a-1)}.
\]

  Taking $n=0$ in the
Borcherds identity \cite[(2.13)]{KMP} gives the (super)commutator formula
\begin{equation}\label{eq:comm}
[a_m,b_n]:=a_mb_n-p(a,b)b_na_m=\sum_{j\in\Z_{\ge 0}}\binom{m+\Delta_a-1}{j}(a_{(j)}b)_{m+n},
\end{equation}
valid on $V$ and on any $\gr$-twisted module, where $\binom{x}{j}=x(x-1)\cdots(x-j+1)/j!$.  

\subsection{The operator $A(z)$ and invariant Hermitian forms}
Let $\phi$ be a conjugate linear involution of a conformal vertex algebra $V$ such that $\phi(L)=L$. Set
\[
g(a)=e^{-\pi\ii\left(\Delta_a+\frac12 p(a)\right)}\phi(a),
\]
so that $g^2=I$, $g(L)=L$ and $gL_1=-L_1g$ \cite[\S3--4]{KMP}.  Put
$A(z)=e^{zL_1}z^{-2L_0}g$; if $a$ is quasiprimary, i.e.\ $L_1a=0$, then
$A(z)a=z^{-2\Delta_a}g(a)$, and $g(a)$ is quasiprimary as well.
A $\Gamma/\Z$-grading $\gr$ is said to be \emph{compatible with $\phi$} if
$\phi(V^{[\gamma]})\subseteq V^{[\gamma]}$ for all $\gamma$ \cite[\S3]{KMP}; then
$g(V^{[\gamma]})\subseteq V^{[\gamma]}$ as well.

All Hermitian forms below are \emph{conjugate linear in the first argument} and
linear in the second.  A Hermitian form $(\cdot,\cdot)$ on $V$ is
\emph{$\phi$-invariant} if
\begin{equation}\label{eq:inv}
(v,Y(a,z)u)=(Y(A(z)a,z^{-1})v,u),\qquad a,u,v\in V
\end{equation}
\cite[Def.~4.1]{KMP}; the same formula, with $Y$ replaced by $Y^M$ on both sides,
defines $\phi$-invariance of a Hermitian form on a $\gr$-twisted $V$-module $M$
\cite[Def.~6.4]{KMP}.  Comparing coefficients in \eqref{eq:inv} yields the mode form
of the invariance condition,
\begin{equation}\label{eq:modeinv}
(v,a_nu)=\Bigl(\sum_{t\in \Z_{\ge0}}\tfrac1{t!}\bigl(L_1^tg(a)\bigr)_{-n}v,\;u\Bigr),
\end{equation}
which for quasiprimary $a$ reduces to \cite[(4.4)]{KMP}
\begin{equation}\label{eq:qpinv}
(v,a_nu)=(g(a)_{-n}v,u).
\end{equation}
For a $\phi$-invariant Hermitian form on a $\gr$-twisted module $M$, with $\gr$
compatible with $\phi$, formulas \eqref{eq:modeinv} and \eqref{eq:qpinv} hold for
$a\in V^{[\gamma]}$, $n\in\gamma-\Delta_a+\Z$ and $u,v\in M$ (cf.\
\cite[(3.18)]{KMP}).

We repeatedly use the following results of \cite{KMP}.

\begin{theorem}[{\cite[Thm.~4.3, Cor.~4.7, Rmk.~4.4]{KMP}}]\label{thm:kmp1}
Let $V$ be a conformal vertex algebra and $\phi$ a conjugate linear involution.
\begin{enumerate}
\item The space of $\phi$-invariant Hermitian forms on $V$ is isomorphic to the space
of conjugate linear functionals $F\in V_0^{\dagger}$ vanishing on $L_1V_1$ and
satisfying $\langle F,g(v)\rangle=\overline{\langle F,v\rangle}$ on $V_0$.  In
particular a non-zero such Hermitian form exists if and only if $L_1V_1=0$, and then it is
unique once normalized by $(\one,\one)=1$.
\item The kernel of a $\phi$-invariant Hermitian form is an ideal of $V$.
\item For quasiprimary $a$ one has $(a,a)=\bigl\langle g(a)_{\Delta_a}a_{-\Delta_a}\one\bigr\rangle
=\bigl\langle g(a)_{(2\Delta_a-1)}a\bigr\rangle$, that is
\begin{equation}\label{eq:top}
g(a)_{(2\Delta_a-1)}a=(a,a)\,\one .
\end{equation}
\end{enumerate}
\end{theorem}
A conformal vertex algebra $V$ with a normalized $\phi$-invariant Hermitian form   is called \emph{unitary} if it  is 
positive semi-definite (in \cite[Def.~4.8]{KMP} the form is required to be
positive definite; the two definitions agree if $V$ is simple). The radical of
the normalized   Hermitian form is the maximal proper ideal of $V$.
Therefore, if $V$ is unitary, the Hermitian form   is positive definite on the unique simple quotient of $V$.

\begin{proposition}[{\cite[Lem.~4.5 and (4.8)]{KMP}}]\label{prop:kmp2}
Let $V$ be a conformal vertex algebra with $L_1V_1=0$.  With respect to the
$sl_2$-action $f=L_{-1}$, $x=-L_0$, $e=-\frac12L_1$ one has
\[
V=\C\one\oplus\bigoplus_i M_i,
\]
where each $M_i$ is a Verma module over $sl_2$ whose highest weight vector $q_i$ is quasiprimary
of conformal weight $\Delta_i>0$.  Equivalently, writing
$Q_\Delta=\ker\bigl(L_1|_{V_\Delta}\bigr)$ for $\Delta>0$,
\begin{equation}\label{eq:sl2dec}
V=\C\one\oplus\bigoplus_{k\ge0}L_{-1}^kQ,\quad \text{where } Q=\bigoplus_{\Delta>0}Q_\Delta ,
\end{equation}
and $L_{-1}$ acts injectively on $\bigoplus_{\Delta>0}V_\Delta$.
\end{proposition}

\subsection{Positive energy modules} Recall that 
a $\gr$-twisted $V$-module $M$ is a \emph{positive energy} module if
it has an $\R$-grading $M=\bigoplus_{j\ge0}M_j$ with $a_nM_j\subseteq M_{j-n}$; then
\begin{equation}\label{eq:pos}
a_nM_0=0\quad (n>0),\qquad a_0M_0\subseteq M_0
\end{equation}
\cite[Def.~2.6, (2.16)]{KMP}.  
A positive energy $\gr$-twisted $V$-module is called \emph{unitary} (with respect
to $\phi$) if it carries a non-zero positive semi-definite $\phi$-invariant Hermitian form.

\begingroup
\subsection{Heisenberg vertex algebras and screening operators}\label{sub:screen}
Let $\h$ be a finite-dimensional vector space with a non-degenerate symmetric
bilinear form $(\cdot|\cdot)$, let $\nu:\h^*\to\h$ be the induced isomorphism,
and let $t\in\C^\times$. We denote by $V^t(\h)$ the Heisenberg vertex algebra
generated by $\h$ with $\l$-bracket $[u_\l v]=\l\,t(u|v)$. For $u\in\h$ we write
$Y(u,z)=u(z)=\sum_{n\in\Z}u_{(n)}z^{-n-1}$, so that
$[u_{(m)},v_{(n)}]=m\,t(u|v)\delta_{m,-n}I$, $u,v\in\h$ (the elements of $\h$ have
conformal weight $1$ for all the conformal vectors considered below, hence
$u_{(n)}=u_n$ in the notation introduced at the beginning of this section). Let $F$ be a conformal vertex algebra with
conformal vector $L^F$, and let $\xi\in\h$. Then
\begin{equation}\label{Lxi}
L_\xi=\frac1{2t}\sum_j:u^ju_j:+\frac1t\,T\xi+L^F,
\end{equation}
where $\{u^j\}$, $\{u_j\}$ are dual bases of $\h$, is a conformal vector of
$V^t(\h)\otimes F$. We write $Y(L_\xi,z)=\sum_{m\in\Z}L_mz^{-m-2}$; explicitly
$L_m=L^{Sug}_m-\frac{m+1}t\,\xi_{(m)}+L^F_m$, where $L^{Sug}$ is the first summand of
\eqref{Lxi}. The free field realizations \eqref{FFR4} and \eqref{s:FFR} of the
principal $W$-algebras are of this form, with $\xi=\tilde\rho$, the vector
defined in \eqref{trho} below (and at the beginning of Section \ref{sec:super}
in the super case).

 For $\mu\in\h^*$ let $M(\mu)$
be the irreducible module of $V^t(\h)$ generated by a vector $|\mu\rangle$ such
that $u_{(n)}|\mu\rangle=0$ for $n>0$ and $u_{(0)}|\mu\rangle=\mu(u)|\mu\rangle$; thus
$M(0)=V^t(\h)$.
 We call $\mu$ the \emph{momentum} of
$M(\mu)$. Let
$s_\mu:M(0)\to M(\mu)$ be the linear map which commutes with the $u_{(n)}$, $n\ne0$,
and maps $\one$ to $|\mu\rangle$, and set
\begin{equation}\label{s:vo}
e^\mu(z)=s_\mu\exp\Bigl(-\sum_{n<0}\frac{\nu(\mu)_{(n)}}{nt}z^{-n}\Bigr)
\exp\Bigl(-\sum_{n>0}\frac{\nu(\mu)_{(n)}}{nt}z^{-n}\Bigr):\ M(0)\to M(\mu)((z)).
\end{equation}
Then $[u_{(n)},e^\mu(z)]=\mu(u)z^ne^\mu(z)$ for all $n\in\Z$, i.e.\
$u(z)e^\mu(w)\sim\frac{\mu(u)}{z-w}e^\mu(w)$, and
$e^\mu(z)\one=|\mu\rangle+O(z)$; these properties characterize $e^\mu(z)$. It is
the intertwining operator of type $\binom{M(\mu)}{M(\mu)\ M(0)}$ attached to
$|\mu\rangle$. More generally, for $\Phi\in F$ we consider the field
\begin{equation}\label{s:Sfield}
S(z)=e^\mu(z)\otimes Y(\Phi,z)=\sum_{m\in\Z}S_{(m)}z^{-m-1}:\ 
V^t(\h)\otimes F\to(M(\mu)\otimes F)((z)),
\end{equation}
and we say that it is a \emph{field of momentum
$\mu$}. The operator $L_0$ acts on $|\mu\rangle$ by the scalar
\begin{equation}\label{s:Deltaxi}
\Delta_\xi(\mu)=\frac{(\mu|\mu)}{2t}-\frac{\mu(\xi)}{t}.
\end{equation}

\begin{lemma}\label{lem:scr}
Let $\mu\in\h^*$ and let $\Phi\in F$ be a primary vector of conformal weight
$\Delta_\Phi$ for $L^F$, with $\mu\ne0$ if $\Phi\in\C\one$. Let $S(z)$ be the
field \eqref{s:Sfield}, and put
$$
\Delta_S=\Delta_\xi(\mu)+\Delta_\Phi,\qquad
\mathcal Q_S=\operatorname{Res}_zS(z)=S_{(0)}:\ V^t(\h)\otimes F\to M(\mu)\otimes F.
$$
Then $S(z)$ is a primary field of conformal weight $\Delta_S$ for $L_\xi$, and
\begin{equation}\label{s:LQ}
[L_m,\mathcal Q_S]=(m+1)(\Delta_S-1)\,S_{(m)}\quad(m\in\Z),\qquad
\mathcal Q_S(L_\xi)=(\Delta_S-1)\,L_{-1}\bigl(|\mu\rangle\otimes\Phi\bigr),
\end{equation}
and $L_{-1}(|\mu\rangle\otimes\Phi)\ne0$. In particular the following are equivalent:

 {\rm(i)} $\mathcal Q_S$ commutes
with $L_m$ for all $m$; 

{\rm(ii)} $\mathcal Q_S(L_\xi)=0$; 

{\rm(iii)}
$\Delta_S=1$.
\end{lemma}
\begin{proof}
Put $v=|\mu\rangle\otimes\Phi$. For $m>0$ each term of $L^{Sug}_m|\mu\rangle$
contains a mode $u_{(n)}$ with $n>0$, and $\xi_{(m)}|\mu\rangle=0$, while
$L^{Sug}_0|\mu\rangle=\frac1{2t}\sum_j\mu(u^j)\mu(u_j)|\mu\rangle$ and
$\xi_{(0)}|\mu\rangle=\mu(\xi)|\mu\rangle$. Since $\Phi$ is primary, we get
$L_mv=0$ for $m>0$ and $L_0v=\Delta_Sv$. Now $S(z)=\mathcal Y(v,z)$ for an
intertwining operator $\mathcal Y$ of $V^t(\h)\otimes F$-modules, so that
\begin{align*}
S(z)\one&=e^{zL_{-1}}v,\\
[L_m,S(z)]&=\sum_{j\ge0}\binom{m+1}{j}z^{m+1-j}\mathcal Y(L_{j-1}v,z)
=z^{m+1}\partial_zS(z)+(m+1)\Delta_Sz^mS(z);
\end{align*}
here we used that $T\xi$ does not contribute to $L_{-1}$, so that $L_{-1}$ is the
translation operator. This says that $S(z)$ is primary of conformal weight
$\Delta_S$, and taking residues gives the first formula in \eqref{s:LQ}. Moreover
$S_{(m)}\one=0$ for $m\ge0$ and $S_{(-2)}\one=L_{-1}v$, hence
$\mathcal Q_S(L_\xi)=\mathcal Q_SL_{-2}\one=-[L_{-2},\mathcal
Q_S]\one=(\Delta_S-1)L_{-1}v$. Finally
$L_{-1}v=\frac1t\nu(\mu)_{(-1)}|\mu\rangle\otimes\Phi+|\mu\rangle\otimes T\Phi\ne0$,
because either $\mu\ne0$ or $T\Phi\ne0$. The equivalence of (i)--(iii) follows.
\end{proof}

\begin{definition}\label{def:scr}
If $\Delta_S=1$ we say that $S(z)$ is a \emph{screening current} and that
$\mathcal Q_S$ is a \emph{screening operator} for $L_\xi$.
\end{definition}

Since $\mathcal Q_S$ is the residue of an intertwining operator, it satisfies
$\mathcal Q_S(a_{(n)}b)=(\mathcal Q_Sa)_{(n)}b+p(S,a)\,a_{(n)}\mathcal Q_Sb$,
where the $n$-th products on the right hand side are those defining the module
structure of $M(\mu)\otimes F$. Hence $\operatorname{Ker}\mathcal Q_S$
is a vertex subalgebra of $V^t(\h)\otimes F$, and it contains $L_\xi$ if and only
if $\mathcal Q_S$ is a screening operator for $L_\xi$.

 We shall need the behaviour of momenta under two
kinds of maps.

(1) Let $t'\in\C^\times$ and $a\in\C^\times$ with $a^2t=t'$. Then $u\mapsto au$,
$u\in\h$, extends to an isomorphism of vertex algebras
$\theta:V^{t'}(\h)\otimes F\to V^t(\h)\otimes F$ which is the identity on $F$. If
$E(w)$ is a field of momentum $\mu$ of $V^{t'}(\h)$, then
$(au)(z)\,\theta E(w)\theta^{-1}\sim\frac{\mu(u)}{z-w}\theta E(w)\theta^{-1}$.
Hence $\theta$ carries fields of momentum $\mu$ at level $t'$ to fields of
momentum $\mu/a$ at level $t$; in particular it conjugates $e^\mu(z)$ at level
$t'$ to a multiple of $e^{\mu/a}(z)$ at level $t$.

(2) Let $\si_\h$ be a conjugate linear bijection of $\h$ such that
$t(\si_\h u|\si_\h v)=\ov{t(u|v)}$ for $u,v\in\h$. Then $\si_\h$ extends to a
conjugate linear automorphism $\si$ of $V^t(\h)$, which induces conjugate linear
isomorphisms $M(\mu)\to M(\si\cdot\mu)$, where
$(\si\cdot\mu)(u)=\ov{\mu(\si_\h^{-1}u)}$, and conjugates $e^\mu(z)$ to a
multiple of $e^{\si\cdot\mu}(z)$. Since
$t(\si_\h\nu(\mu)|u)=\ov{t(\nu(\mu)|\si_\h^{-1}u)}=\ov t\,(\si\cdot\mu)(u)$, we
have, identifying $\h^*$ with $\h$ via $\nu$,
\begin{equation}\label{s:momentum}
\si\cdot\mu=t/\ov t\ \si_\h(\mu).
\end{equation}
\endgroup

\section{Proof of Theorem \ref{11}}
Recall that $\Gamma=\frac{1}{2}\Z$.
\begin{lemma}\label{lem:key}
Let $V$ be a conformal vertex algebra with $L_1V_1=0$, let $\phi$ be a conjugate
linear involution, and let $\gr$ be a $\Gamma/\Z$-grading compatible with $\phi$.
Assume that some positive energy $\gr$-twisted module $M$ carries a non-zero positive
semi-definite $\phi$-invariant Hermitian form $(\cdot,\cdot)_M$.
Let $(\cdot,\cdot)$ be the normalized $\phi$-invariant Hermitian form of $V$,
which exists by Theorem \ref{thm:kmp1}(1). Then
\[
(a,a)\ \ge\ 0\qquad\text{for every quasiprimary }a\in V .
\]
\end{lemma}

\begin{proof}
Let $a$ be quasiprimary of conformal weight $\Delta$. We may assume that $a$ is
parity homogeneous: indeed $L$ is even, so that the even and the odd component of
$a$ are quasiprimary, and they are orthogonal to each other, because by
\eqref{eq:modeinv} their scalar product is of the form $(w,\one)$ with $w$ an odd
vector of $V_0=\C\one$, i.e.\ $w=0$.
If $\Delta=0$ then $a\in\C\one$ and there is nothing to prove, so let
$\Delta>0$ and put $b:=g(a)$, which is quasiprimary of the same conformal
weight and parity as $a$.

Write $a=a^{[0]}+a^{[1/2]}$ and $b=b^{[0]}+b^{[1/2]}$ with
$a^{[\gamma]},b^{[\gamma]}\in V^{[\gamma]}$. Since $L_0$ preserves the grading,
$a^{[\gamma]}$ and $b^{[\gamma]}$ have conformal weight $\Delta$ (but they need not
be quasiprimary), and $b^{[\gamma]}=g(a^{[\gamma]})$, since $\gr$ is compatible
with $\phi$. The modes of $a^{[\gamma]}$ and $b^{[\gamma]}$ on $M$ are indexed by
the class $[\gamma-\Delta]=\gamma-\Delta+\Z\in\Gamma/\Z$, and
$[\gamma-\Delta]=-[\gamma-\Delta]$ because $2(\gamma-\Delta)\in\Z$. As $a$ is
quasiprimary, the invariance of $(\cdot,\cdot)_M$ reads
$(v,Y^M(a,z)u)_M=z^{-2\Delta}(Y^M(b,z^{-1})v,u)_M$. Expanding both sides, the
summand with index $\gamma$ contains only powers $z^{e}$ with $e\in[\gamma]$: on
the left the exponents are $-n-\Delta$ with $n\in[\gamma-\Delta]$, on the right
they are $m-\Delta$ with $m\in[\gamma-\Delta]$, and $2\gamma,2\Delta\in\Z$. Hence
the identity holds for each $\gamma$ separately, and comparing the coefficients of
$z^{-n-\Delta}$, $n\in[\gamma-\Delta]$, we obtain
\begin{equation}\tag{$\ast$}\label{eq:star}
(v,a^{[\gamma]}_nu)_M=(b^{[\gamma]}_{-n}v,u)_M,\qquad u,v\in M,\ n\in[\gamma-\Delta] .
\end{equation}

Since $(\cdot,\cdot)_M$ is positive semi-definite and non-zero, there is a
homogeneous vector $m\in M_{j_0}$ such that $\|m\|_M^2>0$: otherwise, by the
Cauchy--Schwarz inequality, all the homogeneous vectors, hence all the vectors of
$M$, would lie in the kernel of the form. Fix $\gamma$, and let $n\in[\gamma-\Delta]$,
$n>j_0$. By \eqref{eq:star} applied with $v=a^{[\gamma]}_{-n}m$, $u=m$,
\[
\|a^{[\gamma]}_{-n}m\|_M^2=(b^{[\gamma]}_{n}a^{[\gamma]}_{-n}m,\,m)_M .
\]
Since $b^{[\gamma]}_nm\in M_{j_0-n}=0$, we have
$b^{[\gamma]}_na^{[\gamma]}_{-n}m=[b^{[\gamma]}_n,a^{[\gamma]}_{-n}]m$, and
\eqref{eq:comm} gives
\[
[b^{[\gamma]}_n,a^{[\gamma]}_{-n}]=\sum_{j\in\Zp}\binom{n+\Delta-1}{j}\,(b^{[\gamma]}_{(j)}a^{[\gamma]})_{0}.
\]
Now $b^{[\gamma]}_{(j)}a^{[\gamma]}$ has conformal weight $2\Delta-1-j\in\Z$,
so that its modes on $M$ are indexed by $\Z$, and $V_s=0$ for $s<0$; therefore the sum truncates at $j=2\Delta-1$, and
$b^{[\gamma]}_{(2\Delta-1)}a^{[\gamma]}=c_\gamma\one$ for some $c_\gamma\in\C$. Hence
\begin{equation}\label{eq:poly}
0\le\|a^{[\gamma]}_{-n}m\|_M^2=\sum_{j=0}^{2\Delta-1}\binom{n+\Delta-1}{j}\,
\bigl((b^{[\gamma]}_{(j)}a^{[\gamma]})_0\,m,\,m\bigr)_M=:P_\gamma(n).
\end{equation}
The right hand side is a polynomial in $n$ of degree at most $2\Delta-1$, because
$\binom{n+\Delta-1}{j}$ has degree $j$, and the coefficient of
$\binom{n+\Delta-1}{2\Delta-1}$ is $(c_\gamma m,m)_M=\ov{c_\gamma}\,\|m\|_M^2$. Since
$P_\gamma(n)$ is real and non-negative for all $n\in[\gamma-\Delta]$, $n>j_0$, this
coefficient is real and
\[
\ov{c_\gamma}\,\|m\|_M^2=\lim_{n\to\infty}\frac{P_\gamma(n)}{\binom{n+\Delta-1}{2\Delta-1}}\ \ge\ 0 ,
\]
whence $c_\gamma\ge0$, because $\|m\|_M^2>0$.

Finally, by \eqref{eq:top},
$(a,a)\one=b_{(2\Delta-1)}a=\sum_{\gamma,\gamma'}b^{[\gamma]}_{(2\Delta-1)}a^{[\gamma']}$.
The terms with $\gamma\ne\gamma'$ lie in $V_0\cap V^{[1/2]}=0$, so that
$(a,a)=c_0+c_{1/2}\ge0$.
\end{proof}

The following is a slightly more general  statement of Theorem \ref{11}.

\begin{theorem}\label{thm:main}
Let $V$ be a  conformal vertex algebra with $L_1V_1=0$, let $\phi$ be a
conjugate linear involution of $V$ and let $\gr$ be a $\Gamma/\Z$-grading compatible
with $\phi$.  If some positive energy $\gr$-twisted $V$-module carries a
non-zero positive semi-definite $\phi$-invariant Hermitian form, then the normalized
$\phi$-invariant Hermitian form $(\cdot,\cdot)$ of $V$ is positive semi-definite; that is,
$V$ is unitary.
\end{theorem}

\begin{proof}
By Theorem \ref{thm:kmp1}(1) the normalized Hermitian form $(\cdot,\cdot)$ on $V$ exists and is
unique.  We have to  show that it is positive semi-definite.

Since $g(L)=L$, formula \eqref{eq:qpinv} applied to $a=L$ gives
\begin{equation}\label{eq:Ladj}
(v,L_nu)=(L_{-n}v,u),\qquad v,u\in V,
\end{equation}
whence $(L_0v,u)=(v,L_0u)$ and distinct $L_0$-eigenspaces are orthogonal. 

Consider the decomposition \eqref{eq:sl2dec} of Proposition \ref{prop:kmp2}.  Let
$q\in Q_\Delta$, $q'\in Q_{\Delta'}$ and $k>k'\ge0$ with $\Delta+k=\Delta'+k'$
(otherwise the two vectors lie in different $L_0$-eigenspaces and are orthogonal).
Moving the $L_{-1}$'s from the second argument to the first by \eqref{eq:Ladj}, we obtain
\[
(L_{-1}^kq,\,L_{-1}^{k'}q')=(L_1^{k'}L_{-1}^{k}q,\,q')
=\mu\,(L_{-1}^{k-k'}q,\,q')=\mu\,(q,L_1^{k-k'}q')=0
\]
for a suitable scalar $\mu$, because $L_1q'=0$ and $k-k'\ge1$.  Hence the
decomposition \eqref{eq:sl2dec} is orthogonal.  Finally, for $q\in Q_\Delta$ with
$\Delta>0$, the $sl_2$-relation $L_1L_{-1}^kq=k(2\Delta+k-1)L_{-1}^{k-1}q$ gives
\[
(L_{-1}^kq,\,L_{-1}^kq)=(L_1^kL_{-1}^kq,\,q)
=k!\prod_{s=1}^{k}(2\Delta+s-1)\;(q,q)\;\ge\;0
\]
by Lemma \ref{lem:key}, the scalar being a positive real number.
  Together with
$(\one,\one)=1$ this proves that $(\cdot,\cdot)$ is positive semi-definite on $V$.
\end{proof}

\begin{proposition}\label{previous}  Let $V$ be  a  unitary simple vertex algebra. Then the $L_0$-grading of $V$ is compatible with parity, i.e.\
$(-1)^{p(v)}v=(-1)^{2L_0}v$ for all $L_0$-eigenvectors $v\in V$, hence $V$ is a vertex operator superalgebra in
the sense of \cite[Rmk.~2.2]{KMP}.\end{proposition}
\begin{proof}
Call $v$ homogeneous \emph{incompatible} if $2\Delta_v+p(v)$ is odd.  Since $\phi$ is
a conjugate linear involution preserving $L_0$-eigenspaces and parity, each
$V_\Delta$ is spanned over $\C$ by $\phi$-fixed vectors, and $L_1$ commutes with the
relevant decomposition, so it suffices to prove that every $\phi$-fixed incompatible
quasiprimary vector $a$ vanishes.

Let $a$ be such a vector and $\Delta=\Delta_a$.  Skew-symmetry
$Y(x,z)y=p(x,y)e^{zL_{-1}}Y(y,-z)x$ gives, for $x,y$ of conformal weight $\Delta$,
\begin{align*}
x_{(2\Delta-1)}y&=p(x,y)\sum_{j\ge0}\frac{(-1)^{2\Delta+j}}{j!}
L_{-1}^j\bigl(y_{(2\Delta-1+j)}x\bigr)\\
&=p(x,y)(-1)^{2\Delta}\,y_{(2\Delta-1)}x,
\end{align*}
since $y_{(2\Delta-1+j)}x\in V_{-j}=0$ for $j\ge1$.  Taking $x=y=a$ gives
\[
a_{(2\Delta-1)}a=(-1)^{2\Delta+p(a)}\,a_{(2\Delta-1)}a ,
\]
so $\langle a_{(2\Delta-1)}a\rangle=0$ when $2\Delta+p(a)$ is odd.  As $\phi(a)=a$ we
have $g(a)=e^{-\frac{\pi\ii}{2}(2\Delta+p(a))}a$, and \eqref{eq:top} gives
\[
(a,a)=e^{-\frac{\pi\ii}{2}(2\Delta+p(a))}\bigl\langle a_{(2\Delta-1)}a\bigr\rangle=0 .
\]
Since  the Hermitian form is positive definite,  $a=0$.
\end{proof}

\begin{remark} Combining 
Proposition \ref{previous} with Theorem \ref{thm:main} we get   the analogue, for modules, of the spin--statistics
statements of \cite{AL} and of \cite[Rmk.~7.10]{KMP}: if a simple conformal
vertex algebra $V$ with $L_1V_1=0$ contains an odd non-zero vector of integer
conformal weight (or an even non-zero vector of half-integer conformal weight), then $V$
is not unitary, and it has no non-zero unitary positive energy $\gr$-twisted module, for
any choice of $\phi$ and of $\gr$ compatible with $\phi$.
\end{remark}

\section{Unitarity of principal $W$-algebras}\label{sec:W}
Recall that Theorem \ref{thm:kmp1} applies to every conjugate linear involution of the vertex algebra $W^k(\g,\mathfrak s)$ fixing $L$ \cite[Theorem 7.4(b)]{KMP}.\par
Let $\g$ be a simple Lie algebra, $\h\subseteq\g$ a Cartan subalgebra, and let
$\mathfrak s_{pr}=span\{e_{pr},x_{pr},f_{pr}\}$ be a principal $sl_2$-subalgebra of $\g$. We may assume that $x_{pr}\in \h$, and we have
the $ad(x_{pr})$-eigenspace decomposition $\g=\bigoplus\limits_{j\in\Z}\g_j$, where $\g_0=\h$. Then there is a unique subset of positive roots, compatible with the $ad(x_{pr})$-decomposition (i.e. simple root vectors lie in
$\g_1$), and $x_{pr}=\rho^\vee$, the half of the sum of positive coroots. The Dynkin characteristic of $e_{pr}$, i.e.\ the eigenvalues of $ad(x_{pr})$ on simple root vectors, are
$1$.
Throughout this section $(\cdot|\cdot)$ is the  invariant bilinear form, normalized by the condition $(\alpha|\alpha)=2$ for a long root, $\h_\R$ is the real span of the coroots  (on which $(\cdot|\cdot)$ is positive definite), and $\ell=\rank(\g)$.  Let $\rho$ be half of the sum of positive roots. Since $\rho,\rho^\vee\in\h_\R$ are non-zero we have $\|\rho\|^2>0$, $\|\rho^\vee\|^2>0$ and $(\rho|\rho^\vee)>0$  (since $\a(\rho^\vee)$ is the height of $\a$).  We abbreviate
$$t=k+h^\vee\in\C^\times.$$
\par
Let
\begin{equation}\label{trho}\tilde \rho= t \rho^\vee - \rho.\end{equation}
 By    \cite[Corollary 3.1]{KWfree} we have the following free field realization of the universal principal $W$-algebra $W^k(\g)$ associated to $\mathfrak s_{pr}$:
$$\Psi_t: W^k(\g)\to \mathcal V^k=V^{t}(\h),$$
where $V^t(\h)=V^{k+h^\vee}(\h)$ is  the Heisenberg
vertex algebra generated by $\h$ with $[u_\l v]=\l\,t(u|v)$ and 
 \begin{align}
\label{FFR4}&    \Psi_t( L)= L^{\,Sug}_t \ + \ \tfrac{1}{t} \ T \tilde \rho,\end{align}
 where $L^{\,Sug}_t=\frac{1}{2t}\sum_{j=1}^{\ell} : u^j u_j :$, $\{u^j\},\{ u_j\}$ being a pair of dual bases of $	\mathfrak h$ w.r.t. $(\cdot|\cdot)$.
 \par 
  Assume that $(\tilde\rho|\tilde\rho)\ne0$ and write $\h= \C  \tilde\rho\oplus \tilde\rho^\perp$.
Using the normalization
\begin{equation}\label{norm}a=\frac{1}{\sqrt{ t (\tilde\rho|\tilde\rho)}}\tilde\rho,
\end{equation}
we have $t(a|a)=1$; since the $\l$-bracket of $V^t(\h)$ is $[u_\l v]=\l\,t(u|v)$, this gives
$V^{t }(\h)=V^{1}(\C a)\otimes V^{t }(\tilde\rho^\perp)$. (If $(\tilde\rho|\tilde\rho)=0$, then
\eqref{condition} below obviously holds, and, by the proof of Lemma \ref{lem:region}, $t$ is an
endpoint of the arc described there. If $\g$ is not simply laced, this case plays no role in the
sequel; if $\g$ is simply laced, then $t=1$, so that $\tilde\rho=0$ and
$\Psi_t(L)=L^{\,Sug}_t$, and in this case we take for $a$ any element of $\h_\R$ with $(a|a)=1$, and replace $\tilde\rho^\perp$ by
$a^\perp$; then $\l_t=0$ below, and all the arguments of this section go through verbatim.)
The conformal  vector of the vertex algebra $\mathcal V^k$ is
\begin{align}\label{Lfreefield}
L_{free}&=\frac{1}{2}: a a :+L_{\tilde\rho^\perp}.
\end{align}
where $L_{\tilde\rho^\perp}$ is the conformal vector of the Heisenberg algebra attached to $\tilde\rho^\perp$ (see e.g. \cite[Section 5.2]{KMP}). Hence
\begin{align*}\Psi_t( L)&= L_{free}+   \frac{1}{t}T (\tilde\rho)=\frac{1}{2}: a a :+\frac{1}{t}T (\tilde\rho)+L_{\tilde\rho^\perp}\\&=\frac{1}{2}: a a :+\l_t\,T (a)+L_{\tilde\rho^\perp},\qquad \l_t={\frac{\sqrt{t(\tilde\rho|\tilde\rho)}}{t}},\end{align*}
the square root being the one in \eqref{norm}, so that $\frac1t\tilde\rho=\l_t\,a$ and $\l_t^2=(\tilde\rho|\tilde\rho)/t$.
Assume now
\begin{equation}\label{condition}
\frac{(\tilde\rho|\tilde\rho)}{t}\in\R_{\le 0},
\end{equation}
so that $\l_t\in\sqrt{-1}\,\R$.
Then $M(\l_t)\otimes V^{k+h^\vee}(\tilde\rho^\perp)$ is a representation of $W^k(\g)$ via $\Psi_t$; here $M(s)$ denotes the irreducible module of $V^1(\C a)$ on which $a_{(0)}$ acts by $s$.
If $k\in\R$ then condition \eqref{condition} simply means that $k<-h^\vee$, unless $\tilde\rho=0$.
We shall prove in Lemmas \ref{lem:invol}--\ref{lem:fock} below that $M(\l_t)\otimes V^{t}(\tilde\rho^\perp)$ carries a positive definite Hermitian form which is invariant with respect to a conjugate linear involution of $W^k(\g)$ fixing $L$.
Recall from \cite[(3.3)]{KWfree} that the central charge of $W^k(\g)$ is 
\begin{equation}\label{cc}c_{\g,k}=\ell-\frac{12}{t}||\tilde\rho||^2.\end{equation}
Since
$$
\lim_{k\to-\infty}c_{\g,k}=\lim_{k\to-h^\vee}c_{\g,k}=+\infty
$$
the function $c_{\g,k}$ is bounded below on the interval $(-\infty, -h^\vee)$. Set
$$
m_0=\inf_{-\infty<k< -h^\vee}c_{\g,k}.
$$
Elementary calculations give
$$m_0=\ell+24((\rho|\rho^\vee)+||\rho||\cdot||\rho^\vee||).$$
\begingroup
\subsection{The involution and a unitary free field module}\label{sub:complex}
We keep the notation above.  Bars denote complex conjugation of
$\h$ with respect to $\h_\R$, and of scalars.  Besides the properties of $\Psi$
recalled above we use two facts:
\begin{itemize}
\item  a Hermitian form on a module which
satisfies \eqref{eq:inv} for all $a$ in a strong generating set of $V$ satisfies
it for all $a\in V$;
\item for $\g$ simply laced and $(k+h^\vee)(k'+h^\vee)=1$, Feigin-Frenkel duality \cite{FF} holds in the free field form: if $\theta_t:V^{1/t}(\h)\to V^{t}(\h)$ 
is  the vertex algebra isomorphism defined by  
\begin{equation}\label{cx:theta}
\theta_t(u)= -u/t\ (u\in\h),
\end{equation}
then 
\begin{equation}\label{pMM}
\theta_t\(\Psi_{1/t}(W^{k'}(\g))\)=\Psi_t\(W^{k}(\g)\).
\end{equation}
This holds for every $t\in\C^\times$, not only for generic $k$; see
Proposition \ref{prop:FFall} below.
\end{itemize}

\begin{lemma}\label{lem:region} Set $r_0=\|\rho\|/\|\rho^\vee\|$ and
$\cos\theta_0=(\rho|\rho^\vee)/(\|\rho\|\,\|\rho^\vee\|)\in(0,1]$.  Then
\eqref{condition} holds if and only if either $t\in\R_{<0}$, or $t$ lies on the arc
$t=r_0e^{\ii\theta}$ with $\cos\theta\le\cos\theta_0$.  Moreover $c_{\g,k}$ maps
$\R_{<0}$ onto $[m_0,+\infty)$ and the above arc onto $[\ell,m_0]$, so that
\eqref{condition} describes exactly the locus where $c_{\g,k}$ is real and
$\ge\ell$.  If $\g$ is simply laced, then $\rho=\rho^\vee$, $r_0=1$ and
$\theta_0=0$, so the arc is the whole circle $S^1$.
\end{lemma}
\begin{proof} Write $A=\|\rho^\vee\|^2$, $B=(\rho|\rho^\vee)$, $C=\|\rho\|^2$, all
positive, and let
$$
F(t):=\frac{(\tilde\rho|\tilde\rho)}{t}=tA-2B+\frac{C}{t}.
$$
If $t=re^{\ii\theta}$ then $\Im F=\sin\theta\,(rA-C/r)$, so $F\in\R$ if and only if
$t\in\R^\times$ or $r=\sqrt{C/A}=r_0$.  For $t$ real, $F(t)=\|\tilde\rho\|^2/t$ has
the sign of $t$, because $\tilde\rho\in\h_\R$ and $\tilde\rho\ne0$ unless
$\rho=t\rho^\vee$, which for real $t$ happens only if $\g$ is simply laced and
$t=1$; apart from that point, which lies on the circle $|t|=r_0$, $F\le0$ exactly for
$t<0$, and then $F(t)=-\frac{1}{12}(c(t)-\ell)$ gives, by the
arithmetic-geometric mean inequality, $c(\R_{<0})=[m_0,+\infty)$.  For
$|t|=r_0$,
$$
F(r_0e^{\ii\theta})=\|\rho\|\,\|\rho^\vee\|(e^{\ii\theta}+e^{-\ii\theta})-2B
=2\|\rho\|\,\|\rho^\vee\|\cos\theta-2(\rho|\rho^\vee),
$$
which is real and $\le0$ precisely when $\cos\theta\le\cos\theta_0$; on that arc
$$
c=\ell+24(\rho|\rho^\vee)-24\|\rho\|\,\|\rho^\vee\|\cos\theta
$$
takes all the values between $\ell$ (at $\theta=\pm\theta_0$) and $m_0$ (at
$\theta=\pi$).
\end{proof}

Fix a square root $\sqrt{t(\tilde\rho|\tilde\rho)}$ and set, as in \eqref{norm} (for $\tilde\rho=0$ we use the
convention stated after \eqref{norm}, and then $\si(a)=-a$ in case (2) of Lemma
\ref{lem:invol} below because $a\in\h_\R$),
\begin{equation}\label{cx:alam}
a=\frac{\tilde\rho}{\sqrt{t(\tilde\rho|\tilde\rho)}},\qquad
\l_t={\frac{\sqrt{t(\tilde\rho|\tilde\rho)}}{t}},\qquad\text{so that}\qquad
\Psi_t(L)=\tfrac12:aa:+\l_t\,T(a)+L_{\tilde\rho^\perp},
\end{equation}
and $t(a|a)=1$, $\l_t\in\ii\R$ by \eqref{condition}.  Replacing this square root by its opposite replaces $(a,\l_t)$ by $(-a,-\l_t)$, and leaves $\Psi_t(L)$ and the
$V^1(\C a)$-module $M(\l_t)$ unchanged.

\begin{lemma}\label{lem:invol} Let $\si:\h\to\h$ be a conjugate linear map.  Then $\si$
extends to a conjugate linear automorphism of $V^t(\h)$ fixing $\Psi_t(L)$ if and
only if
\begin{equation}\label{cx:sigcond}
{\rm(a)}\ \ t(\si u|\si v)=\ov{t(u|v)}\ \ (u,v\in\h),
\qquad\qquad
{\rm(b)}\ \ \si(\tilde\rho)=\frac{\ov t}{t}\,\tilde\rho .
\end{equation}
Such a $\si$ exists, and satisfies $\si^2=\mathrm{id}$, $\si(a)=-a$ and
$t(\cdot|\cdot)<0$ on $\h^\si\setminus\{0\}$, in each of the following cases ($u\in \h)$:
\begin{enumerate}
\item $t\in\R_{<0}$, with $\si(u)=\ov u$;
\item $\g$ simply laced and $|t|=1$, with $\si(u)=-\ov t\,\ov u$.
\end{enumerate}
Conversely, if $\si$ satisfies \eqref{cx:sigcond} and $\tilde\rho\ne0$, then
$t\in\R^\times$ or $|t|=r_0$.
\end{lemma}
\begin{proof} Since $\tilde\rho_{(0)}$ acts trivially on the vacuum module,
$\Psi_t(L)$ and $L^{\,Sug}_t$ define the same grading, so an automorphism fixing
$\Psi_t(L)$ preserves the weight one space $\h$ and is determined by its
restriction to it.  Condition (a) says precisely that $\si$ is a (conjugate
linear) automorphism of $V^t(\h)$; granting (a), $\si$ maps
$t(\cdot|\cdot)$-dual bases to $t(\cdot|\cdot)$-dual bases, so
$\si(L^{\,Sug}_t)=L^{\,Sug}_t$, and since
$\si(\frac1tT\tilde\rho)=\frac1{\ov t}T(\si\tilde\rho)$, fixing $\Psi_t(L)$ is
equivalent to (b).

Let $\cdot \mapsto  \ov{(\cdot)}$ denote the complex conjugation of $\g$. Write $\si=\chi M\circ\ov{(\cdot)}$ with $M$ orthogonal for $(\cdot|\cdot)$ and
$\chi^2=\ov t/t$, which is what (a) amounts to.  Condition (b) reads
$\chi\(\ov tM\rho^\vee-M\rho\)=\frac{\ov t}{t}\(t\rho^\vee-\rho\)$.  Taking the
square of the norms of both sides and using that $M$ is orthogonal, the terms in
$(\rho|\rho^\vee)$ cancel and one is left with
$$
\frac{\ov t}{t}\(\ov t^{\,2}A+C\)=\ov t^{\,2}A+\frac{\ov t^{\,2}}{t^2}C
\iff (t-\ov t)\(C-|t|^2A\)=0,
$$
that is $t\in\R^\times$ or $|t|=r_0$.  This proves the converse statement.
\par
(1) For $t$ real, $\si=\ov{(\cdot)}$ satisfies (a) because $t$ is real and
$(\cdot|\cdot)$ is defined over $\h_\R$, and (b) because $\tilde\rho\in\h_\R$ and
$\ov t/t=1$.  \par(2) For $\g$ simply laced and $|t|=1$ we have
$\tilde\rho=(t-1)\rho$ with $\rho\in\h_\R$; conditions \eqref{cx:sigcond} for
$\si(u)=c\ov u$ read $tc^2=\ov t$ and $c(\ov t-1)=\frac{\ov t}{t}(t-1)$, and
$(t-1)/(\ov t-1)=-t$ gives $c=-\ov t$, which indeed satisfies $c^2=\ov
t^{\,2}=\ov t/t$.

In both cases $|c_t|=1$ with $c_t=1$, resp.\ $c_t=-\ov t$, so $\si_t^2=\mathrm{id}$;
writing $c_t=e^{2\ii\psi}$ we get $\h^{\si_t}=e^{\ii\psi}\h_\R$ and
$t(u|u)=t\,c_t(w|w)$ for $u=e^{\ii\psi}w$, $w\in\h_\R$, which is negative because
$t\,c_t=t<0$ in case (1) and $t\,c_t=-t\ov t=-1$ in case (2).  Finally
$a=\zeta\tilde\rho$ with $\zeta=(t(\tilde\rho|\tilde\rho))^{-1/2}$ and
$\tilde\rho\in\h_\R$ in case (1), $\tilde\rho=(t-1)\rho$ in case (2); in case (1)
$t(\tilde\rho|\tilde\rho)<0$, so $\zeta$ is purely imaginary and
$\si(a)=\ov\zeta\,\tilde\rho=-\zeta\tilde\rho=-a$, while in case (2)
$a=\zeta'\rho$ with $\zeta'=1/(\sqrt t\,\|\rho\|)$, up to the choice of the square root, so
$\zeta'/\ov{\zeta'}=\ov t=c_t\cdot(-1)$ and
$\si(a)=c_t\ov{\zeta'}\rho=-\zeta'\rho=-a$ as well.
\end{proof}

\begin{lemma}\label{lem:miura} Denote by $\si_t$  the automorphisms of $V^t(\h)$ displayed  in  cases {\rm(1)}, {\rm(2)} of Lemma
\ref{lem:invol}. Then  $\si_t\(\Psi_t(W^k(\g))\)=\Psi_t(W^k(\g))$, so that $\si_t$
induces a conjugate linear involution $\phi_t$ of $W^k(\g)$ with $\phi_t(L)=L$
and $\Psi_t\circ\phi_t=\si_t\circ\Psi_t$.
\end{lemma}
\begin{proof} In case (1), $\si_t=\ov{(\cdot)}$ is the conjugation with respect to
the split real form of $V^t(\h)$, and it is intertwined by $\Psi_t$ with the conjugate linear involution of
$W^k(\g)$ induced by the split real form involution of $\g$; in particular it
preserves the image (cf.\ \cite[Lemma 7.3]{KMP}).

In case (2), $\ov t=1/t$ and $\ov k=k'$ is the Feigin-Frenkel dual level.  By the
same token, complex conjugation $\ov{(\cdot)}:V^t(\h)\to V^{1/t}(\h)$ satisfies
$\ov{\Psi_t(W^k(\g))}=\Psi_{1/t}(W^{k'}(\g))$, and since
$\theta_t(\ov u)=-\ov u/t=-\ov t\,\ov u=\si_t(u)$, \eqref{pMM}, which holds for
every $t\in\C^\times$ by Proposition \ref{prop:FFall}, gives
$$
\si_t\(\Psi_t(W^k(\g))\)=\theta_t\(\Psi_{1/t}(W^{k'}(\g))\)=\Psi_t\(W^k(\g)\).
$$
By Lemma \ref{lem:invol}, $\si_t$ fixes $\Psi_t(L)$ and is an involution; as
$\Psi_t$ is injective, the last assertion follows.  (Consistently, $\theta_t$
preserves the conformal vectors: the data \eqref{cx:alam} attached to $1/t$ are
$a'=\sqrt t\,\rho/\|\rho\|$ and $\l_{1/t}=-\l_t$, so $\theta_t(a')=-a'/t=-a$ and
$\theta_t(\Psi_{1/t}(L))=\frac12:aa:+\l_tT(a)+L_{\tilde\rho^\perp}=\Psi_t(L)$.)
\end{proof}
The following result is essentially proved in \cite{KMP} and \cite{KMP1}. We provide a proof for completeness.
\begin{lemma}\label{lem:fock} In both cases of Lemma \ref{lem:invol}, the module
$$
M_k=M(\l_t)\otimes V^{k+h^\vee}(\tilde\rho^\perp),
$$
is a positive energy module over $\(\mathcal V^k,\Psi_t(L)\)$, and the Hermitian
form $H=H_1\otimes H_2$, where $H_1$ is the $\si_t$-invariant Hermitian form on   $M(\l_t)$ over $V^1(\C a)$ normalized by $H_1(v_{\l_t},v_{\l_t})=1$ and
$H_2$ is the $\si_t$-invariant Hermitian form on $V^{k+h^\vee}(\tilde\rho^\perp)$
normalized by $H_2(\one,\one)=1$, is positive definite.
\end{lemma}
\begin{proof} By Lemma \ref{lem:invol} and $\Delta_a=1$ we get
$g(a)=e^{-\pi\ii}\si_t(a)=a$, while $g(u)=-u$ for
$u\in\h^{\si_t}\cap\tilde\rho^\perp$.  From \eqref{cx:alam}, $L_1a=-2\l_t\one$
and $L_1u=0$, so \eqref{eq:modeinv} implies (recall that $a_n=a_{(n)}$ and
$u_n=u_{(n)}$, since $a$ and $u$ have conformal weight $1$)
\begin{equation}\label{cx:adj}
(v,a_{(n)}w)=(a_{(-n)}v,w)-2\ov{\l_t}\,\d_{n,0}(v,w),\qquad
(v,u_{(n)}w)=-(u_{(-n)}v,w).
\end{equation}
Both factors of $\mathcal V^k$ carry their own conformal vector, namely
$\frac12:aa:+\l_tT(a)$ and $L_{\tilde\rho^\perp}$, and $\si_t$ preserves each of
them; so, by the discussion of tensor products at the end of
\cite[Section 4]{KMP}, $H=H_1\otimes H_2$ is $\si_t$-invariant as soon as $H_1$
and $H_2$ are.  Both factors are  discussed in the literature.

The factor $H_1$ is \cite[Proposition 6.3]{KMP1}: the irreducible module $M(\mu)$ over
the Heisenberg vertex algebra generated by $\h=\C a$,  $(a,a)=1$,  with Fairlie conformal vector $\frac12:aa:+sT(a)$ carries a
non-zero invariant Hermitian form if and only if $s=\ii\,\Im(\mu)$, which, for $s=\l_t$, is our
condition $\mu\in\l_t+\R$, and \cite[(6.13)]{KMP1} computes it on the monomial
basis, giving in particular its positive definiteness.  Here $\mu=\l_t\in\ii\R$
by \eqref{condition}, in accordance with the fact that the $n=0$ case of
\eqref{cx:adj} reads $\l_t=-\ov{\l_t}$.

The factor $H_2$ is \cite[Section 5.2]{KMP}: $V^{k+h^\vee}(\tilde\rho^\perp)$ is
the vertex algebra of superbosons attached to the purely even space
$\tilde\rho^\perp$ with the bilinear form $t(\cdot|\cdot)$, on which $\si_t$
restricts to a conjugate linear involution; by loc.\ cit.\ the invariant form
exists, the monomial basis is orthogonal, and it is positive definite if and only
if $t(\cdot|\cdot)$ is positive definite on
$\{u\in\tilde\rho^\perp\mid\si_t(u)=-u\}=\ii\,(\h^{\si_t}\cap\tilde\rho^\perp)$,
i.e.\ if and only if $t(\cdot|\cdot)$ is negative definite on
$\h^{\si_t}\cap\tilde\rho^\perp$.  This latter fact is given by Lemma \ref{lem:invol}.  Explicitly, for
an orthogonal basis $u^{(1)},\dots,u^{(\ell-1)}$ of the real vector space
$\h^{\si_t}\cap\tilde\rho^\perp$, the ordered monomials
$a_{(-n_1)}\cdots u^{(j_1)}_{(-m_1)}\cdots v_{\l_t}$ ($n_i,m_r>0$) form an orthogonal
basis of $M_k$; the squared norm of such a monomial is a positive multiple of
$\prod_in_i\cdot\prod_r\(-m_r\,t(u^{(j_r)}|u^{(j_r)})\)$.
\par
Finally
$L_0=\frac12a_{(0)}^2-\l_ta_{(0)}+\sum_{j>0}a_{(-j)}a_{(j)}+(\tilde\rho^\perp\text{-part})$, so
the minimal energy is $-\frac12\l_t^2=\frac12|\l_t|^2\ge0$, the spectrum lies in
$\frac12|\l_t|^2+\Z_{\ge0}$ and all eigenspaces are finite dimensional.
\end{proof}

\endgroup

\begin{theorem}\label{4.1} Let $\g$  be a simple Lie algebra. Then the principal $W$-algebra $W^k(\g)$ is  unitary  for the central charge
\begin{equation}\label{ef}c_{\g,k}\geq m_0=\ell+24((\rho|\rho^\vee)+\|\rho\|\cdot\|\rho^\vee\|).\end{equation}\end{theorem}
\begingroup
\begin{proof} By Lemma \ref{lem:region}, $c_{\g,k}\ge m_0$ means that $t\in\R_{<0}$,
and then \eqref{condition} holds.  Set $V=\Psi_t(W^k(\g))\subseteq\mathcal V^k$, a
conformal vertex algebra with conformal vector $\Psi_t(L)$, $V_0=\C\one$ and, since 
$V_1=0$, $L_1V_1=0$.  By Lemma \ref{lem:miura}, $\si_t$ restricts to a conjugate linear
involution of $V$ fixing $\Psi_t(L)$, so by Theorem \ref{thm:kmp1}(1) $V$ has a
unique normalized $\si_t$-invariant Hermitian form; by Lemma \ref{lem:fock},
$M_k$ is a non-zero positive energy $V$-module with a positive definite
$\si_t$-invariant Hermitian form.  Theorem \ref{thm:main}, applied with $\gr$
trivial, shows that $V\cong W^k(\g)$ is unitary.   Here we use the injectivity of the map $\Psi$.
\end{proof}
\endgroup
\begin{theorem}\label{4.2} Let $\g$  be a simply laced simple Lie algebra. Then $W_k(\g)$ is a unitary vertex algebra  for
\begin{equation}\label{ef2}c_{\g,k}\geq 	\rank(\g).\end{equation}
\end{theorem}
\begingroup
\begin{proof} Let $c_{\g,k}$ be real and $\ge\ell=\rank(\g)$.  By Lemma
\ref{lem:region} (with $r_0=1$, $\theta_0=0$) we have $t\in\R_{<0}\cup S^1$ and
\eqref{condition} holds.  If $t\in\R_{<0}$ this is Theorem \ref{4.1}.  If
$|t|=1$, case (2) of Lemmas \ref{lem:invol}, \ref{lem:miura} and \ref{lem:fock}
applies, and the proof of Theorem \ref{4.1} goes through verbatim.  Since $c_{\g,k}$
maps $S^1$ onto the interval
\begin{equation}\label{dadim}	[\ell, m_0]\end{equation}
the unit circle fills exactly the gap left by Theorem \ref{4.1}.
\end{proof}

For $\g=sl_3$ one has $\ell=2$ and $m_0=98$, and Theorem \ref{4.2} is the
unitarity of the vacuum representation of the $W_3$-algebra for $c\ge2$, proved
in \cite{CTW}. The proof in \cite{CTW} also relies on a modification of the free
field realization of Fateev and Zamolodchikov \cite{FZ}, which yields a representation
that is unitary on a suitable invariant subspace; the range $2\le c\le98$
treated there is the interval \eqref{dadim} filled by the unit circle $|t|=1$.

\endgroup
\begin{remark}\label{rem:ACL}\
If $\mathfrak{g}$ is a simply laced simple Lie algebra and $k$ is a
non-negative integer, then by \cite{ACL} the commutant of the diagonally
embedded affine vertex algebra $V_{k+1}(\mathfrak g)$ in $V_1(\mathfrak{g})\otimes V_k(\mathfrak{g})$ is
\begin{equation}\label{eq:ACL}
\operatorname{Com}\bigl(V_{k+1}(\mathfrak{g}),\,V_1(\mathfrak{g})\otimes V_k(\mathfrak{g})\bigr)\;\cong\;
W_{c(k)}(\mathfrak{g}),
\end{equation}
and $V_{k+1}(\mathfrak{g})\otimes W_{c(k)}(\mathfrak{g})\subset V_1(\mathfrak{g})\otimes V_k(\mathfrak{g})$
is a conformal embedding; here $W_{c(k)}(\mathfrak{g})$ denotes the simple principal $W$-algebra with
central charge
\begin{equation}\label{eq:ck}
c(k)=\operatorname{rank}(\mathfrak{g})
\left(1-\frac{h^\vee(h^\vee+1)}{(k+h^\vee)(k+h^\vee+1)}\right).
\end{equation}
Hence these principal $W$-algebras are unitary, being commutants of a unitary subalgebra in the unitary vertex algebra $V_1(\g)\otimes V_k(\g)$ .
It is expected
that these, along with the principal $W$-algebras with central charge
$\ge\operatorname{rank}(\mathfrak{g})$, are all the unitary principal
$W$-algebras. Of course, this is well-known for $\g=sl_2$.
 \end{remark}
The following proposition shows that, if $\g$ is not simply laced, the free
field construction of this section cannot reach the range $\rank(\g)<c<m_0$, at
least with involutions mapping $\h_\R$ onto a multiple of itself.
\begin{proposition}\label{rem:nonsl} Let $\g$ be a simple Lie algebra, let $t\notin\R$ satisfy
\eqref{condition} with $(\tilde\rho|\tilde\rho)/t<0$, and let $\si$ be a
conjugate linear involution of $V^t(\h)$ satisfying \eqref{cx:sigcond} and such
that $\si(\h_\R)=c\,\h_\R$ for some $c\in\C^\times$.  If the Hermitian
form $H_2$ of Lemma \ref{lem:fock} is positive definite, then
$\rho\in\R\rho^\vee$, i.e.\ $\g$ is simply laced.
\end{proposition}
(The hypothesis on $\si$ holds in both cases of Lemma \ref{lem:invol},
with $c=1$ for $t$ real and $c=-\ov t$ for $|t|=1$.)
\begin{proof}We first check that $\si(a)=-a$, which follows from
\eqref{cx:sigcond} and the strict inequality in \eqref{condition} alone.  Put
$D=(k+h^\vee)(\tilde\rho|\tilde\rho)=t^2F$, where $F=(\tilde\rho|\tilde\rho)/t<0$,
so that $\sqrt D=t\sqrt F$ with $\sqrt F\in\ii\R$ and $\ov{\sqrt D}=-\ov t\sqrt F$.
Since $a=\tilde\rho/\sqrt D$ by \eqref{norm} and $\si$ is conjugate linear,
$$
\si(a)=\ov{\Bigl(\frac{1}{\sqrt D}\Bigr)}\,\si(\tilde\rho)
=\frac{1}{-\ov t\sqrt F}\cdot\frac{\ov t}{t}\,\tilde\rho
=-\frac{\tilde\rho}{t\sqrt F}=-\frac{\tilde\rho}{\sqrt D}=-a .
$$
Since $\si(a)=-a$ we have $\ii a\in\h^\si$ and
$t(\ii a|\ii a)=-t(a|a)=-1$, and $\h^\si=\R(\ii a)\oplus(\h^\si\cap\tilde\rho^\perp)$;
so $H_2$ is positive definite if and only if $t(\cdot|\cdot)$ is negative definite
on the whole of $\h^\si$.

Write $\si=c\,B\circ\ov{(\cdot)}$ with $B=c^{-1}\si\circ\ov{(\cdot)}$, so
that $B(\h_\R)=\h_\R$; rescaling $c$ we may assume $|c|=1$.  Being the composite of two  linear maps, $B$ is complex
linear, and it preserves $\h_\R$, hence is defined over $\R$.  For $u,v\in\h_\R$
condition \eqref{cx:sigcond}(a) reads $t\,c^2(Bu|Bv)=\ov{t(u|v)}=\ov t\,(u|v)$,
so that $B$ is a dilation of $(\h_\R,(\cdot|\cdot))$ of ratio
$\mu=\frac{\ov t}{tc^2}$, which is real because $B$ is.  Moreover $\si^2=id$
amounts to $B^2=id$, whence $\mu^2=1$; and $\mu=-1$ is impossible, $(\cdot|\cdot)$
being positive definite on $\h_\R$.  Therefore $\mu=1$, that is
$c^2=\frac{\ov t}{t}$ and $B\in O(\h_\R)$.  Let
$\h_\R=\h_\R^+\oplus\h_\R^-$ be the eigenspace decomposition of $B$ and put
$c=e^{2\ii\gamma}$.  Then $u\in\h^\si$ if and only if
$u=e^{\ii\gamma}(w_++\ii w_-)$ with $w_\pm\in\h_\R^\pm$, and
$$
t(u|u)=t\,c\,\bigl[(w_+|w_+)-(w_-|w_-)\bigr].
$$
As $(\cdot|\cdot)$ is positive definite on $\h_\R$, the bracket takes both signs
unless $\h_\R^-=0$ or $\h_\R^+=0$, i.e.\ unless $B=\pm id$.  In either case one
finds $tc<0$, resp.\ $tc>0$, and in both cases
$$
\h^\si=\frac{\ii}{\sqrt t}\,\h_\R .
$$
Finally \eqref{cx:sigcond}(b) says that $\xi:=\frac{\sqrt t}{\ii}\tilde\rho$
satisfies $\ov\xi=\frac{\ov t}{t}\xi$, i.e.\ that $\tilde\rho$ is a complex
multiple of a vector of $\h_\R$.  Writing $\tilde\rho=t\rho^\vee-\rho$ with
$\rho,\rho^\vee\in\h_\R$, this means that $\Im(t)\,\rho^\vee$ and
$\Re(t)\rho^\vee-\rho$ are proportional; as $\Im(t)\ne0$, we get
$\rho\in\R\rho^\vee$.
\end{proof}

\section{The super case}\label{sec:super}\ 
In this section $\g$ is a basic Lie superalgebra with $\g_{\ov 1}\ne0$.
See \cite{KAdv} for the related notation and results.
A nilpotent element in  $\g_{\ov 0}$ is called {\it even} if the
eigenvalues of $2 ad\,x$ on $\g_{\ov 0}$ (resp. $\g_{\ov 1}$)  are even (resp
odd) integers.   
Any even nilpotent element determines a compatible set of positive roots, and its Dynkin characteristic is uniquely determined.
(The Dynkin characteristic is the ordered set of  eigenvalues of $ad\, x$
on simple root vectors; they take values $0, 1/2,$ or $1$, cf. \cite[Theorem 1.2]{EK}.)
Indeed, even vertices of the Dynkin characteristic on which the label is $0$ 
are those simple root vectors on which $ad\, x$ is $0$,
and they are simple root vectors of a reductive Lie algebra $\g_0$;
odd vertices are the lowest weight vectors of the representation of $\g_0$
on $\g_{1/2}$, where $ad\, x$ has eigenvalue $1/2$, and eigenvalues $1$ are on the lowest
weight vectors of $\g_0$ on
$\g_1/[\g_{1/2},\g_{1/2}]$.

Let $\mathfrak s_{pr}=span\{e_{pr},x_{pr},f_{pr}\}\subseteq\g_{\ov 0}$ be  a principal
$sl_2$-subalgebra of $\g_{\ov 0}$, where $x=x_{pr}=\rho^\vee$, the half sum of all positive coroots of $\g_{\ov 0}$. The nilpotent $f_{pr}$ is  even in all cases except for $\g=A(m,n)$ with $m-n$ even, which we exclude from consideration. In all other cases  the set of simple roots $\Pi_x$ is the unique set of simple roots compatible
with the $ad(x)$-grading, so that
\begin{equation}\label{s:values}
\a(x)=1\ (\a\in\Pi_x \text{ even}),\qquad \a(x)=\tfrac12\ (\a\in\Pi_x\text{ odd}).
\end{equation}
Their Dynkin characteristics in all cases are described in \cite[Section 6]{KWfree}.
We have  $\g_0=\h$; by \cite[\S 6.2]{KWfree} this excludes $A(m,n)$ with $m-n$
even.  We keep the notation $t=k+h^\vee$ and
$\tilde\rho=t\rho^\vee-\nu(\rho)$, and  $\rho=\rho_{\ov 0}-\rho_{\ov 1}$ where $\rho_{\ov 0}$ (resp. $\rho_{\ov 1}$) is the   half   sum of positive even (resp. odd) roots 
and $\nu:\h^*\to \h$ is the isomorphism induced by the invariant bilinear form on $\g$. The latter is  normalized 
as in \cite[Section 5.3]{KMP}, by choosing an even highest root $\theta$ of $\g$
and requiring
\begin{equation}\label{s:norm}
(\theta|\theta)=2 ,
\end{equation}
which for  a simple Lie algebra $\g$ is the normalization of Section \ref{sec:W}.
With this choice $h^\vee$, defined by $\kappa_\g=2h^\vee(\cdot|\cdot)$, and the
level $k$ are those of \cite{KRW}, \cite{KWfree}, \cite{KMP}, \cite{KMP1}, whose
formulas we quote.  For instance $(\d_i|\d_j)=\frac12\delta_{ij}$ and
$h^\vee=n+\frac12$ for $osp(1|2n)$; $(\d_i|\d_j)=\frac12\delta_{ij}$,
$(\e|\e)=-\frac12$ and $h^\vee=n$ for $osp(2|2n)$; $(\d|\d)=-(\e|\e)=\frac12$ and
$h^\vee=\frac12$ for $spo(2|3)$.

In this generality the free field realization of Section \ref{sec:W} reads, by
\cite[Corollary 3.1]{KWfree},
\begin{equation}\label{s:FFR}
\Psi:W^k(\g)\to\mathcal V^k=V^{t}(\h)\otimes F(\g_{1/2}),\qquad
\Psi(L)=\frac{1}{t}\Big(\frac12\sum_{j=1}^{\dim\h}:u^ju_j:+T\tilde\rho\Big)+L^F,
\end{equation}
where $F(\g_{1/2})$ is the vertex algebra of free fermions based on $\g_{1/2}$
with the bilinear form $\langle a,b\rangle=(f_{pr}|[a,b])$ and
$L^F=\frac12\sum_{\g_\a\subseteq\g_{1/2}}:T(\Phi_\a)\Phi_\a:$ is its conformal
vector; the central charge is
\begin{equation}\label{s:cc}
c_{\g,k}=c_{free}-\frac{12}{t}(\tilde\rho|\tilde\rho),\quad\text{where }
c_{free}=\rank(\g)+\tfrac12\dim\g_{1/2},
\end{equation}
and the module of Lemma \ref{lem:fock} becomes
$$
M_k=M(\l_t)\otimes V^{k+h^\vee}(\tilde\rho^\perp)\otimes F(\g_{1/2}),\qquad
H=H_1\otimes H_2\otimes H_3,
$$
$H_3$ being the $L^F$-invariant Hermitian form on $F(\g_{1/2})$.  

\begin{remark}\label{s:remF} The factor $F(\g_{1/2})$ plays little role in what
follows, for the following reason.  Its conformal vector
$L^F=\frac12\sum_{\g_\a\subseteq\g_{1/2}}:T(\Phi_\a)\Phi_\a:$ and the bilinear form
$\langle a,b\rangle=(f_{pr}|[a,b])$ do not depend on $k$; hence conditions (a),
(b) of Lemma \ref{lem:invol}, where the whole discussion below takes place,
involve only $\si|_\h$, and $F(\g_{1/2})$ enters only through the discrete
requirement that $\langle\si(\cdot),\cdot\rangle$ be negative definite on
$\g_{1/2}$, i.e.\ that $H_3$ be positive definite. \end{remark}

\begin{lemma}\label{s:lem1} Condition \eqref{condition} holds if and only if
$c_{\g,k}$ is real and $c_{\g,k}\ge c_{free}$.
\end{lemma}
\begin{proof} Immediate from $c_{\g,k}=c_{free}-\frac{12}{t}(\tilde\rho|\tilde\rho)$.
\end{proof}

What has to be examined is the second half of the
argument, namely Lemmas \ref{lem:invol}--\ref{lem:fock}.  For each value of
$t$ one needs a conjugate linear involution $\si_t$ of $\mathcal V^k$ such that
$\si_t(\Psi_t(L))=\Psi_t(L)$, $\si_t$ preserves $\Psi_t(W^k(\g))$, and the
Hermitian form $H$ is positive definite; as in Section \ref{sec:W}, $\si_t$ may
depend on $t$.  By Remark \ref{s:remF} we take $\si_t=\si_\h\otimes\si_F$ on
$V^t(\h)\otimes F(\g_{1/2})$, with $\si_\h$ satisfying conditions (a), (b) of
Lemma \ref{lem:invol}.

\begingroup
\begin{remark}\label{s:remscreen}
Conditions (a), (b) in \eqref{cx:sigcond} and the negativity of $t(\cdot|\cdot)$ on $\h^\si$ can be met
for every $\g$ and every $t<0$ satisfying \eqref{condition} strictly: take the
conjugation with respect to a real form of $\h$ that contains $\tilde\rho$ and on
which $(\cdot|\cdot)$ is positive definite. So the relevant constraint is that $\si_t$
preserve $\Psi_t(W^k(\g))$. We explain, in terms of screening operators, which
involutions can be expected to do so. We use the notation of Section
\ref{sub:screen}, with $F=F(\g_{1/2})$ and $\xi=\tilde\rho$, so that
$L_\xi=\Psi_t(L)$ by \eqref{s:FFR}. Since
$\mu(\tilde\rho)=t\mu(x)-(\mu|\rho)$, \eqref{s:Deltaxi} reads
\begin{equation}\label{s:Deltamu}
\Delta(\mu):=\Delta_{\tilde\rho}(\mu)=\frac{(\mu|\mu)}{2t}-\mu(x)+\frac{(\mu|\rho)}{t}.
\end{equation}

 Let $\a\in\Pi_x$ and
$\mu\in\h^*$, $\mu\ne0$. Consider the field 
$$S(z)=\begin{cases}e^\mu(z)\quad&\text{if $\a$ is even,}\\ 
:e^\mu(z)\Phi_\a(z):\quad&\text{if $\a$ is odd,}\end{cases}$$ where $\Phi_\a$ is the field of
$F(\g_{1/2})$ attached to a root vector $e_\a\in\g_{1/2}$; it is primary of
conformal weight $\frac12$ for $L^F$. By Lemma \ref{lem:scr} and
\eqref{s:values}, $S(z)$ is a screening current for $\Psi_t(L)$ if and only if
$\Delta(\mu)=\a(x)$. Since $2(\rho|\a)=(\a|\a)$ for a simple root $\a$,
\eqref{s:Deltamu} gives
$$
\Delta(c\,\a)-\a(x)=(c+1)\Bigl(\frac{c\,(\a|\a)}{2t}-\a(x)\Bigr)\qquad(c\in\C).
$$
So on the line $\C\a$ there are exactly two momenta giving a screening current:
\begin{equation}\label{s:twomom}
\mu_\a=-\a\quad\text{for all }t,\qquad\text{and}\qquad
\mu^\vee_\a=\frac{2t\,\a(x)}{(\a|\a)}\,\a\quad\text{if $\a$ is not isotropic}.
\end{equation}
We denote by $S_\a(z)$, $\mathcal Q_\a$ the screening current and the screening
operator attached to $\mu_\a$, and by $S^\vee_\a(z)$, $\mathcal Q^\vee_\a$ those
attached to $\mu^\vee_\a$, which we call the \emph{dual} ones; when the level
has to be specified we write $\mathcal Q_\a^{(t)}$, $\mathcal Q_\a^{\vee,(t)}$.

The $\mathcal Q_\a$ are the screening
operators of \cite[Thm.~3.9]{G}, which in our notation reads
\begin{equation}\label{s:screen}
\Psi_t\bigl(W^k(\g)\bigr)=\bigcap_{\a\in\Pi_x}\operatorname{Ker}\mathcal Q_\a
\qquad\text{for generic $k$.}
\end{equation}
For every $t\in\C^\times$ the left hand side is contained in the right hand
side, since both depend algebraically on $t$ (see the proof of Proposition
\ref{prop:FFall}), but the inclusion may be strict. Assume now that
$\rho=\l\rho^\vee$ with $\l\ne0$. By Proposition \ref{s:prop2} below, no simple root
is isotropic and $\mu^\vee_\a=\frac t\l\a$ for all $\a\in\Pi_x$. Let $t'=\l^2/t$
and let $\theta_t$ be the isomorphism \eqref{FF:sigma} below. Since
$\theta_t(u)=-\frac\l tu$, by Section \ref{sub:screen}\,(1) it carries $S_\a$ at
level $t'$, of momentum $-\a$, to a multiple of $S^\vee_\a$ at level $t$, of
momentum $\frac t\l\a$. Hence
\begin{equation}\label{s:thetaQ}
\theta_t\Bigl(\bigcap_{\a\in\Pi_x}\operatorname{Ker}\mathcal Q_\a^{(t')}\Bigr)
=\bigcap_{\a\in\Pi_x}\operatorname{Ker}\mathcal Q_\a^{\vee,(t)} .
\end{equation}
Thus Feigin--Frenkel duality, in the form of Proposition
\ref{prop:FFall}(2) below, says that for generic $k$ the two families of screening
operators cut out the same subalgebra:
\begin{equation}\label{s:dualscreen}
\Psi_t\bigl(W^k(\g)\bigr)=\bigcap_{\a\in\Pi_x}\operatorname{Ker}\mathcal Q_\a
=\bigcap_{\a\in\Pi_x}\operatorname{Ker}\mathcal Q^\vee_\a .
\end{equation}

 Let $\si_\h$ satisfy (a) of Lemma
\ref{lem:invol}, so that Section \ref{sub:screen}\,(2) applies. We say that
$\si_t=\si_\h\otimes\si_F$ \emph{permutes the screening operators} if there is a
permutation $\tau$ of $\Pi_x$, preserving the parities and $(\cdot|\cdot)$, such
that $\si_t$ conjugates each $S_\a$, $\a\in\Pi_x$, to a non-zero multiple of
$S_{\tau\a}$ or of $S^\vee_{\tau\a}$. By \eqref{s:screen} and
\eqref{s:dualscreen}, for generic $k$ such an involution preserves
$\Psi_t(W^k(\g))$. We claim that it has the form
\begin{equation}\label{s:sigform}
\si_\h=c\,\tau\circ\ov{(\cdot)},\qquad |c|=1,\quad \tau\in\operatorname{Aut}(\Pi_x).
\end{equation}
Indeed, write $\si_\h=B\circ\ov{(\cdot)}$ with $B$ complex linear. The simple
roots are real, so by \eqref{s:momentum} $\si_t$ maps a field of momentum $-\a$
to a field of momentum $-\frac t{\ov t}B\a$. Comparing with \eqref{s:twomom}, we
get $B\a=\eta_\a\,\tau\a$, where
$$
\eta_\a=\frac{\ov t}{t}\quad\text{if }S_\a\mapsto S_{\tau\a},\qquad
\eta_\a=-\ov t\,\frac{2\a(x)}{(\a|\a)}\quad\text{if }S_\a\mapsto S^\vee_{\tau\a}.
$$
On real vectors (a) of \eqref{cx:sigcond} reads: $t(Bu|Bv)=\ov t(u|v)$; as $\tau$ preserves
$(\cdot|\cdot)$, this gives $\eta_\a\eta_\be=\ov t/t$ whenever $(\a|\be)\ne0$.
If some simple root $\a$ is not isotropic, then $\eta_\a^2=\ov t/t$ and
$\eta_\be=\eta_\a$ for every neighbour $\be$ of $\a$ in the Dynkin diagram, hence
for all $\be\in\Pi_x$, the diagram being connected. If all simple roots are
isotropic only the first case occurs, and again all the $\eta_\a$ coincide.
So $B=c\,\tau$ with $c^2=\ov t/t$, which is \eqref{s:sigform}. The argument also
shows that either $c=1$ and $t\in\R$ (each $S_\a$ is mapped to $S_{\tau\a}$), or
$\frac{2\a(x)}{(\a|\a)}=\frac1\l$ does not depend on $\a$, i.e.\
$\rho=\l\rho^\vee$, and $c=-\ov t/\l$, $|t|=|\l|$ (each $S_\a$ is mapped to
$S^\vee_{\tau\a}$).

We call the involutions of the form \eqref{s:sigform} \emph{of diagram type},
and we restrict to them in Proposition \ref{s:prop4} below; they include the
involutions of Lemma \ref{lem:invol} and of Theorems \ref{s:thm5} and
\ref{s:thm7}. Note that the computation of momenta is valid for every
$t\in\C^\times$; genericity enters only through \eqref{s:screen}. In what
follows, the fact that $\si_t$ preserves $\Psi_t(W^k(\g))$ is never deduced from
\eqref{s:screen} at a level which is not known to be generic: for $t\in\R$ it
follows from the split real form, as in Lemma \ref{lem:miura}, and on the
circles $|t|=\l$ from Proposition \ref{prop:FFall}, which extends
\eqref{s:dualscreen}, in the form $\theta_t(\Psi_{t'}(W^{k'}(\g)))=\Psi_t(W^k(\g))$,
to all non-critical levels.
\end{remark}
\endgroup

We first determine when $\rho$ and $\rho^\vee$ are proportional.

\begin{proposition}\label{s:prop2} The following are equivalent for a basic Lie superalgebra  $\g\ne A(m,n), \ m-n$ even:
\begin{enumerate}
\item $\rho\in\C\rho^\vee$;
\item there is $\l\in\C$ such that $(\a|\a)=2\l$ for every even $\a\in\Pi_x$ and
$(\a|\a)=\l$ for every odd $\a\in\Pi_x$;
\end{enumerate}
and then $\rho=\l\rho^\vee$ and $\tilde\rho=(t-\l)\rho^\vee$.  In particular
$\rho=0$ if and only if $\Pi_x$ is totally isotropic.
\end{proposition}
\begin{proof} For every simple root of a basic Lie superalgebra one has
$2(\rho|\a)=(\a|\a)$, while $(\rho^\vee|\a)=\a(x)$ is given by \eqref{s:values}.
Hence $\rho=\l\rho^\vee$ forces $(\a|\a)=2\l\a(x)$, which is (2); conversely (2)
gives $(\rho-\l\rho^\vee|\a)=0$ for all $\a\in\Pi_x$, whence $\rho=\l\rho^\vee$
by nondegeneracy of the invariant bilinear form.
\end{proof}

\begin{proposition}\label{s:prop3} Let $\g$ be a basic Lie superalgebra, except for $A(m,n),\ m-n$ even.  Then $\rho=\lambda \rho^\vee$ for some $\l\in\C$ exactly in the following cases:
$$
\begin{array}{lll}
\l=1: & \g\text{ is a simply laced simple Lie algebra,}\\[2pt]
\l>0: &\g=osp(1|2n)\ (n\ge1), \l=\tfrac12,\\[2pt]
\l=0: &\g=sl(n+2|n+1),\ osp(2n|2n),\ osp(2n+2|2n),\ D(2,1;a).
\end{array}
$$
\end{proposition}
\begin{proof} If $\Pi_x$ contains an isotropic root then $\l=0$ by Proposition
\ref{s:prop2}, hence every $\a\in\Pi_x$ is isotropic and $\rho=0$; the simple
systems consisting of isotropic roots are classified, and the principal one is of
this kind exactly for the algebras listed in the third line.  For instance for
$\g=sl(m+1|n+1)$ one has $\e_i(x)=\frac m2+1-i$, $\d_j(x)=\frac n2+1-j$, so that
the roots of value $\frac12$ are $\e_i-\d_j$ with $i-j=\frac{m-n-1}{2}$ and
$\d_j-\e_i$ with $i-j=\frac{m-n+1}{2}$; these span $\h^*$ if and only if
$|m-n|=1$, and then every even root of value $1$ is a sum of two of them.  If
$m-n\ge3$ the root $\e_1-\e_2$ is an even element of $\Pi_x$ and (2) fails.  The
same computation for $osp(2m|2n)$, where $\e_i(x)=m-i$ and $\d_k(x)=n-k+\frac12$,
gives a totally isotropic $\Pi_x$ exactly for $m-n\in\{0,1\}$; for $D(2,1;a)$,
$\Pi_x$ consists of the three isotropic roots of value $\frac12$.

If no root of $\Pi_x$ is isotropic, then $\g$ has no isotropic roots at all, so $\g=osp(1|2n)$.  There $\Pi_x=\{\d_1-\d_2,\dots,\d_{n-1}-\d_n,\d_n\}$
with $(\d_i|\d_j)=\frac12\delta_{ij}$ for the normalized form, so
$(\a|\a)=1$ for the even simple roots and $(\d_n|\d_n)=\frac12$: (2) holds with
$\l=\frac12$.  In all the remaining cases $\Pi_x$ contains both an isotropic odd
root and either an even root or a non-isotropic odd root: e.g.\ for
$\g=spo(2|3)=osp(3|2)$ one finds $\Pi_x=\{\d,\ \e-\d\}$ with $(\d|\d)=-(\e|\e)=\frac12$ by \eqref{s:norm}, so that $(\d|\d)=\frac12$ and
$(\e-\d|\e-\d)=0$; for $F(4)$, $\Pi_x=\{\e_2-\e_3\}\cup\{\text{three isotropic
roots}\}$; for $G(3)$, $\Pi_x=\{\e_2-\e_1,\ \d,\ \e_1-\d\}$; for
$C(n)=osp(2|2n-2)$, $n\ge3$, $\Pi_x$ contains $\d_1-\d_2$.
\end{proof}

Assume $\rho=\l\rho^\vee$, so that $\tilde\rho^{\,\perp}=x^\perp$ does not depend
on $t$ and, by Proposition \ref{s:prop2},
$$
\frac{(\tilde\rho|\tilde\rho)}{t}=\frac{(t-\l)^2}{t}(x|x).
$$
If $(x|x)>0$ this is $\le0$ exactly for $t<0$ and, when $\l>0$, on the circle
$|t|=\l$, where it equals $2(\Re t-\l)(x|x)$. In the cases of Proposition
\ref{s:prop3} with $\l\ne0$ one has $\rho(x)>0$, hence, $\rho(x)=\l(x|x)$ being
normalization-free, $(x|x)$ and $\l$ are both positive for the normalization
\eqref{s:norm}; if $\l=0$ and $(x|x)<0$ (which may happen since for $\l=0$ the sign of
$(x|x)$ depends on the choice of $\theta$ in \eqref{s:norm}), the discussion applies
after replacing $t$ by $-t$.   So the picture of Section
\ref{sub:complex} is reproduced verbatim, with the circle $|t|=\l$ in place of
$S^1$ and with Feigin--Frenkel duality in the form $tt'=\l^2$; and $c$ runs over
$[m_{free},+\infty)$ on $\R_{<0}$ and over $[c_{free},m_{free}]$ on the circle, where now
$$
m_{free}=c_{free}+48\,\l\,(x|x).
$$
For $\l=0$ no circle is needed: $t<0$ already gives all the values
$c>c_{free}$.\par
At this point the parameter side of the construction is settled for the
four families of Proposition \ref{s:prop3}: condition \eqref{condition} holds on
$\R_{<0}$, and on the circle $|t|=\l$ when $\l>0$, and $c$ covers the whole range
$c\ge c_{free}$.  For real $t$, on the other hand, no proportionality is
needed.  In all cases what is left is the positivity of the Hermitian form, which
by the following proposition amounts to the existence of a diagram automorphism
correcting the signature of $(\cdot|\cdot)$ on $\h_\R$.
\begin{proposition}\label{s:prop4} Let $t$ satisfy \eqref{condition} with strict
inequality, and let $\si=\si_\h\otimes\si_F$ be a conjugate linear involution of
$\mathcal V^k$ fixing $\Psi_t(L)$, with $\si_\h=c\,\tau\circ\ov{(\cdot)}$,
$|c|=1$, and $\tau$ an automorphism of the Dynkin diagram of $\Pi_x$.  Then
$\tau$ fixes $x$ and $\rho$, $\tau^2=id$, and:
\begin{enumerate}
\item if $t\in\R$ then $c=1$; if $t\notin\R$ then $\rho=\l\rho^\vee$ with $\l>0$,
$|t|=\l$ and $c=-\ov t/\l$;
\item $H_2$ is positive definite if and only if $t\,c\,Q_\tau$ is negative
definite on $\h_\R$, where $Q_\tau(w_++w_-)=(w_+|w_+)-(w_-|w_-)$ for
$w_\pm\in\h_\R^{\tau=\pm1}$;
\item in particular, if $t<0$ or $|t|=\l$, then $t\,c<0$, and $H_2$ is positive
definite if and only if $(\cdot|\cdot)$ is positive definite on $\h_\R^{\tau=1}$
and negative definite on $\h_\R^{\tau=-1}$.
\end{enumerate}
Among the basic Lie superalgebras with $\g_{\ov 1}\ne0$ and $\g_0=\h$, the
condition in {\rm(3)} holds for some $\tau$ exactly for $\g=osp(1|2n)$, with
$\tau=id$, and for $\g=osp(2|2n)$, $n\ge1$, with $\tau$ the automorphism
$\e\mapsto-\e$ induced by $O(2)$, which exchanges the two odd simple roots
$\d_n\pm\e$.
\end{proposition}

\begin{proof} Since $\tau$ preserves $\Pi_x$ and the parities, it preserves the
values \eqref{s:values}, hence fixes $x$; it permutes the positive roots, hence
fixes $\rho$; and $\si^2=id$, $|c|=1$ give $\tau^2=id$.  As $\tau$ is an isometry
and $(\cdot|\cdot)$ is real on $\h_\R$, condition (a) of Lemma \ref{lem:invol}
reads $c^2=\ov t/t$, and (b) reads
$c\,(\ov t\rho^\vee-\rho)=\ov t\rho^\vee-\frac{\ov t}{t}\rho$.

(1) For $t$ real, (b) gives $(c-1)\tilde\rho=0$, so $c=1$.  For $t\notin\R$, if
$\rho\notin\C\rho^\vee$ then (b) gives $c=1$ and $c=\ov t/t$, which is absurd,
and the same happens if $\rho=0$; so $\rho=\l\rho^\vee$ with $\l\ne0$,
$c=\frac{\ov t}{t}\,\frac{t-\l}{\ov t-\l}$, and $c^2=\ov t/t$ becomes
$(t-\ov t)(|t|^2-\l^2)=0$, i.e.\ $|t|=|\l|$.  By Proposition \ref{s:prop3}
$\l>0$, and then $c=-\ov t/\l$.

(2) By the computation at the beginning of the proof of Proposition
\ref{rem:nonsl}, $\si(a)=-a$, so that
$\h^\si=\R\,\ii a\oplus(\h^\si\cap\tilde\rho^{\,\perp})$ with $t(\ii a|\ii a)=-1$;
hence $H_2$ is positive definite if and only if $t(\cdot|\cdot)$ is negative
definite on $\h^\si$.  Writing $c=e^{2\ii\psi}$ one has
$\h^\si=e^{\ii\psi}\bigl(\h_\R^{\tau=1}\oplus\ii\,\h_\R^{\tau=-1}\bigr)$, and for
$u=e^{\ii\psi}(w_++\ii w_-)$ one finds $t(u|u)=t\,c\,Q_\tau(w_++w_-)$, since
$(w_+|w_-)=0$.

(3) By (1), $t\,c=t$ for $t$ real and $t\,c=-\l$ on the circle.

For the classification, note that an automorphism of $\Pi_x$ is induced by an
automorphism of $\g$ normalizing $\h$ and fixing $x$, and that, $x$ being regular
in $\g_{\ov 0}$, such an automorphism acts on $\h$ through an outer automorphism
of $\g$.  If $\tau=id$, the condition says that $(\cdot|\cdot)$ is positive
definite on $\h_\R$, which happens only for $osp(1|2n)$: in all the other cases
$\g_{\ov 0}$ has two ideals (simple or one-dimensional) on whose real Cartan
subalgebras the invariant bilinear form takes opposite signs, so that it is
indefinite on $\h_\R$. A non-trivial $\tau$ exists only for $sl(m|n)$, $osp(2m|2n)$ and
$D(2,1;a)$.  For $sl(m|n)$, $\tau$ is the flip $a_i\mapsto-a_{m+1-i}$,
$b_j\mapsto-b_{n+1-j}$ of $h=\mathrm{diag}(a|b)$: if $m,n\ge2$ its
$+1$-eigenspace contains vectors of both signs, while for $n=1$ its
$-1$-eigenspace, on which the form is $\sum_ia_i^2-(\sum_ia_i)^2$ with
$a_i=a_{m+1-i}$, is negative definite only for $m=2$; the case $m=1$ is
symmetric.  For $osp(2m|2n)$, $\tau$ is $\e_m\mapsto-\e_m$ (note that
$\e_m(x)=0$), whose $+1$-eigenspace contains $\e_1,\dots,\e_{m-1}$ and
$\d_1,\dots,\d_n$, of opposite signs, unless $m=1$; for $m=1$ the normalization
$(2\d_1|2\d_1)=2$ gives $(\d_i|\d_j)=\frac12\delta_{ij}$ and $(\e|\e)=-\frac12$,
so that the condition holds.  For $D(2,1;a)$, $\tau$ permutes the $\e_i$
preserving the $s_i$, and a transposition $\e_i\leftrightarrow\e_j$ with
$s_i=s_j$ has $+1$-eigenspace containing $\e_i+\e_j$ and $\e_l$, of opposite signs
since $s_l=-2s_i$.
\end{proof}

\begingroup
In the terminology of Remark \ref{s:remscreen}, the two cases of Proposition
\ref{s:prop4}(1) are distinguished by the action of $\si$ on the momenta
\eqref{s:twomom} of the screening currents. By \eqref{s:momentum},
$\si\cdot\mu= t / \ov t\,\, c\,\tau(\ov\mu)$, so that
\begin{align*}
\si\cdot\mu_\a&=\mu_{\tau\a},& \si\cdot\mu^\vee_\a&=\mu^\vee_{\tau\a}&&(t\in\R,\ c=1),\\
\si\cdot\mu_\a&=\mu^\vee_{\tau\a},& \si\cdot\mu^\vee_\a&=\mu_{\tau\a}&&
(|t|=\l,\ c=-\ov t/\l),
\end{align*}
the dual momenta being considered only for non-isotropic $\a$.
Thus for real $t$ the involution $\si$ permutes the screening currents $S_\a$
among themselves, while on the circle $|t|=\l$ it exchanges them with the dual
currents $S^\vee_\a$. This is why, on the circle, the fact that $\si$ preserves
$\Psi_t(W^k(\g))$ requires Feigin--Frenkel duality, i.e.\ the second equality in
\eqref{s:dualscreen}, whereas on the real line it does not.
\endgroup

\begin{theorem}\label{s:thm5} Let $\g=osp(1|2n)$, $n\ge1$, so that $h^\vee=n+\frac12$, and let $W_k(\g)$ be
its principal simple $W$-algebra (it is  strongly generated by fields of conformal weight $2,4,\ldots,2n,n+\frac{1}{2}$).  Then  $W_k(\g)$ is
unitary for
$$
c_{\g,k}\ \ge\ n+\tfrac12=\rank(\g)+\tfrac12\dim\g_{1/2}.
$$
\end{theorem}

\begin{proof} Here $\Pi_x=\{\d_1-\d_2,\dots,\d_{n-1}-\d_n,\d_n\}$ with
$(\d_i|\d_j)=\frac12\delta_{ij}$ for the normalized invariant bilinear form, so that
$\rho=\sum_i(n-i+\frac12)\d_i=\frac12\rho^\vee$, i.e.\ $\l=\frac12$ in
Proposition \ref{s:prop3}, and $\tilde\rho=(t-\frac12)\rho^\vee$.  Moreover
$\d_i(x)=n-i+\frac12$, so no root vanishes on $x$ and the only root of value
$\frac12$ is $\d_n$: thus $\g_0=\h$, $\g_{1/2}=\C e_{\d_n}$ and
$c_{free}=n+\frac12$.  Finally $(x|x)=(\rho^\vee|\rho^\vee)=\frac{n(4n^2-1)}6$.

\  By \eqref{s:cc}, $c=c_{free}-\frac{12}{t}(t-\frac12)^2(x|x)$.
For $t<0$ the arithmetic-geometric mean inequality gives
$c\ge m_{free}:=c_{free}+48\l(x|x)=n+\frac12+4n(4n^2-1)$, with equality at $t=-\frac12$,
and $c$ runs over $[m_{free},+\infty)$.  On the circle $t=\frac12e^{\ii\theta}$ one has
$c=c_{free}+24\l(x|x)(1-\cos\theta)$, which runs over $[c_{free},m_{free}]$, the value
$c_{free}$ being attained only at $t=\l$, where $\tilde\rho=0$.  So every
$c>c_{free}$ arises from some $t$ with $(\tilde\rho|\tilde\rho)/t<0$.

 Set $\si_t=\ov{(\cdot)}$ for $t<0$ and
$\si_t(u)=-\frac{\ov t}{\l}\ov u$ for $|t|=\l$; in the second case
$t\ov t=\l^2$ gives \eqref{cx:sigcond}(a), and (b) holds because $\rho^\vee$ is
real.  In both cases $\si_t^2=id$, $\si_t(a)=-a$ and $t\,c_t<0$, being $t$,
resp.\ $-\l$.  Since $\g_{\ov 0}=sp(2n)$, the form is positive definite on
$\h_\R$, so $t(\cdot|\cdot)$ is negative definite on $\h^{\si_t}$, as in Lemma
\ref{lem:invol}.

 For $t<0$, $\si_t$ is the conjugation with respect to the
split real form and preserves $\Psi_t(W^k(\g))$ as in Lemma \ref{lem:miura}.  For
$|t|=\l$ one has $\ov t=\frac1{4t}$, so complex conjugation maps
$\Psi_t(W^k(\g))$ onto $\Psi_{1/(4t)}(W^{k'}(\g))$ with
$(k+h^\vee)(k'+h^\vee)=\frac14$, and $\si_t=\theta_t\circ\ov{(\cdot)}$ for
$\theta_t(u)=-u/(2t)$; by Proposition \ref{prop:FFall} (the duality of
\cite{CL}, \cite{G}, extended to all non-critical levels),
$\theta_t\bigl(\Psi_{1/(4t)}(W^{k'}(\g))\bigr)=\Psi_t\bigl(W^k(\g)\bigr)$, so
$\si_t$ preserves $\Psi_t(W^k(\g))$. In the terminology of Remark
\ref{s:remscreen}, $\si_t$ permutes the screening operators with $\tau=id$: for
$t<0$ it maps each screening current $S_\a$ to a multiple of itself, and for
$|t|=\l$ it exchanges $S_\a$, of momentum $-\a$, with the dual current
$S^\vee_\a$, of momentum $2t\a$.

  $F(\g_{1/2})$ is a single neutral free fermion, and the sign
of $\langle a,b\rangle=(f_{pr}|[a,b])$ can be adjusted by composing $\si_t$ with
the parity automorphism of $\g$, which fixes $\g_{\ov 0}$ pointwise, hence
$\Psi_t(L)$ and the Miura image; thus $H_3$ is positive definite.

 By Lemma \ref{lem:fock} the module $M_k$ carries a positive
definite $\si_t$-invariant Hermitian form, and $W^k(\g)_1=0$, the conformal
weights of the generators being $2,4,\dots,2n$ and $n+\frac12$.  As in the proof
of Theorem \ref{4.1}, Theorem \ref{thm:main} gives that $W_k(\g)$ is unitary for every $c>c_{free}$; the value $c=c_{free}$ follows by
Remark \ref{closed} below.
\end{proof}

\begin{theorem}\label{s:thm7} Let $\g=osp(2|2n)$, $n\ge1$, and let $W_k(\g)$ be
its principal simple $W$-algebra (it is strongly generated by fields of conformal weights $1,2,4,\dots,2n$, and $n+\frac12$ (twice)).  Then $W_k(\g)$ is
unitary for all $t=k+h^\vee<0$, that is, for
\begin{align*}
c(W^k(\g)) &\ge \rank(\g)+1+24((\rho|\rho^\vee)+\Vert\rho\Vert \Vert\rho^\vee\Vert 
\\ &=\ n+2+2n(n-1)(4n+1)+2\sqrt2\,n(2n-1)\sqrt{(n-1)(2n+1)}.
\end{align*}
\end{theorem}
\begin{proof} For the normalized form one has $(\d_i|\d_j)=\frac12\delta_{ij}$,
$(\e|\e)=-\frac12$,
$\Pi_x=\{\d_1-\d_2,\dots,\d_{n-1}-\d_n,\ \d_n-\e,\ \d_n+\e\}$,
$\rho^\vee=\sum_i(2n-2i+1)\d_i$ and $\rho=\sum_i(n-i)\d_i$, so that
$$
(x|x)=\tfrac{n(4n^2-1)}6,\qquad (\rho|\rho^\vee)=\tfrac{n(n-1)(4n+1)}{12},\qquad
\|\rho\|^2=\tfrac{n(n-1)(2n-1)}{12},
$$
and $(\tilde\rho|\tilde\rho)=\frac12\sum_i\bigl((2n-2i+1)t-(n-i)\bigr)^2>0$;
hence \eqref{condition} holds strictly for every $t<0$.  Moreover $\g_0=\h$ and
$\g_{1/2}=\g_{\d_n-\e}\oplus\g_{\d_n+\e}$, so that $c_{free}=n+2$.  The
automorphism $\e\mapsto-\e$ of $\h$ is induced by an element of $O(2)$, which
acts on $\g$ fixing $\mathfrak s_{pr}\subseteq sp(2n)$; composed with the
conjugation with respect to the split real form it gives a conjugate linear
involution $\si$ with $\si_\h=\tau\circ\ov{(\cdot)}$, which preserves
$\Psi_t(W^k(\g))$ and, by Proposition \ref{s:prop4}(3), makes $H_2$ positive
definite. In the terminology of Remark \ref{s:remscreen}, $\si$ permutes the
screening operators: it fixes, up to scalars, the screening currents of the even
simple roots and exchanges those of $\d_n-\e$ and $\d_n+\e$; these two roots
being isotropic, they have no dual screening currents, in accordance with the
fact that only real $t$ occurs.  As for $sl(2|1)$, the two root spaces in $\g_{1/2}$ are isotropic,
exchanged by $\tau$ and paired by $\langle a,b\rangle=(f_{pr}|[a,b])$, so that
$F(\g_{1/2})$ is a $bc$-system and $H_3$ is positive definite, possibly after
composing $\si$ with the parity automorphism.  Since $L_1W^k(\g)_1=0$ \cite[Thm.~7.4(b)]{KMP}, the proof of Theorem \ref{4.1}
applies.  Finally, for $t<0$,
$c=c_{free}+24(\rho|\rho^\vee)-12t(x|x)-\frac{12}t\|\rho\|^2$, whose minimum is
$\rank(\g)+\frac12 \dim\g_{1/2}+24((\rho|\rho^\vee)+\|\rho\|\,\|\rho^\vee\|)$.
\end{proof}

\begin{remark}\label{closed} The  answers of Theorem \ref{s:thm5} and Theorem \ref{s:thm7} agree with what is known about unitarity of the corresponding
superconformal algebras (see \cite{KMP1} and the references therein): the Neveu--Schwarz algebra $W_k(osp(1|2))$ is unitary
for $c\ge\frac32$ (Theorem \ref{s:thm5} with $n=1$, $m_{free}=\frac{27}2$). The
$N=2$ algebra is unitary for $c\ge3$  (Theorem
\ref{s:thm7} with $n=1$).  

The value $c=c_{free}$ is not reached by the construction, for two
different reasons.  For $\g=osp(1|2n)$ it corresponds to $t=\l=\frac12$, i.e.\ to
$k=-n$: there $\tilde\rho=0$, so that
\eqref{condition} holds with equality and $a$, $\l_t$ in \eqref{cx:alam} are not
defined (unlike the case of $osp(2|2)$ below, $c_{free}$ is not approached as $t\to0$, since $c\to+\infty$ as $t\to0^-$).  For $\g=osp(2|2)$,
where $\rho=0$ and $c=3-6t$, it corresponds to $t=0$, i.e.\ to the critical
level, at which $W^k(\g)$ is not a conformal vertex algebra. 

In both cases unitarity still holds at $c=c_{free}$, by continuity.  If
$(V^c)$ is a family of conformal vertex algebras, freely generated by fields of
fixed conformal weights, whose $\l$-brackets and conjugate linear involutions
depend continuously on $c$, then the Gram matrices of the corresponding invariant
Hermitian forms depend continuously on $c$; positive semidefiniteness being a
closed condition, the set of $c$ for which $V^c$ is unitary is closed.  \end{remark}

\begingroup
\begin{remark}\label{s:rem7} It follows from the formula for the central charge
of $W_k(\g,\mathfrak s)$ that this conformal vertex algebra and
$W_{k'}(\g,\mathfrak s)$ have the same central charge if
\begin{equation}\label{scc} (k+h^\vee)(k'+h^\vee)=\frac{\|\rho\|^2}{\|x\|^2}.
\end{equation}
Indeed, the equation for $c$ in \cite[(3.3)]{KWfree} is, up to a factor, a monic
quadratic polynomial in $k+h^\vee$ with constant term
$\frac{\|\rho\|^2}{\|x\|^2}$. By the strange formula, $||\rho||^2= (sdim\,\g) h^\vee /12$, 
\begin{equation}
\|\rho\|^2=0 \iff \text{ either $h^\vee=0$ or $\sdim\,\g=0$.}\label{sf}
\end{equation}
In these, and only in these, cases $c$ is a linear function of $k$. In all other
cases $k\ne k'$ in \eqref{scc}, and it is an interesting open problem to determine
when the corresponding $W$-algebras are isomorphic.
For $\mathfrak s=\mathfrak s_{pr}$ one has $x=\rho^\vee$. Feigin--Frenkel duality
shows that the two $W$-algebras are isomorphic when $\g$ is a simple Lie algebra  different from $B_n,C_n$, $n>2$. For
$osp(1|2n)$ the same holds. Genra proves it via the identification of
$W^k(osp(1|2n))$ with the Fateev--Lukyanov algebra \cite[Thm.~6.4,
Rem.~6.5]{G}. Creutzig and Linshaw prove the isomorphism of the
$\Z_2$-orbifolds \cite[Thm.~4.1, \S4.1]{CL}, and deduce the full duality from a
uniqueness theorem for the odd simple current extension \cite[Rem.~4.1,
Thm.~6.3]{CL}. It is also the case $(X,Y)=(O,B)$, $m=0$, of the dualities of
hook-type $\mathcal W$-superalgebras in \cite[Thm.~4.4, Cor.~4.5]{CLNS}. These results are usually
stated for \emph{generic} $k$, i.e.\ for $k$ outside a countable set. However, in case (2) of Lemma
\ref{lem:miura} and in Theorem \ref{s:thm5} the duality is needed, in
the free field form \eqref{pMM}, at the non-real levels of the circles $|t|=1$
and $|t|=\frac12$, and it is not known whether these levels are generic. For
this reason we do not rely there on the generic statements: we prove in
Proposition \ref{prop:FFall} below that the duality holds at every non-critical
level, and it is this proposition that is used in case (2) of Lemma \ref{lem:miura} and
in Theorem \ref{s:thm5}.
\end{remark}

In the two cases needed above we have $\rho=\l\rho^\vee$ with $\l=1$ ($\g$
simply laced) or $\l=\frac12$ ($\g=osp(1|2n)$), see Proposition
\ref{s:prop3}. For $t\in\C^\times$ put
\begin{equation}\label{FF:sigma}
\begin{gathered}
t'=\sigma(t):=\frac{\l^2}{t},\qquad
\theta_t:V^{t'}(\h)\otimes F(\g_{1/2})\to V^{t}(\h)\otimes F(\g_{1/2}),\\
\theta_t(u)=-\frac{\l}{t}\,u\ \ (u\in\h),\qquad \theta_t|_{F(\g_{1/2})}=\mathrm{id}.
\end{gathered}
\end{equation}
The map  $\theta_t$ is an isomorphism of vertex algebras;
for $\l=1$ it is the map \eqref{cx:theta}. Here $F(\g_{1/2})=\C$ if $\g$ is a Lie
algebra, and $F(\g_{1/2})$ is a single neutral free fermion if $\g=osp(1|2n)$.

\begin{proposition}\label{prop:FFall}
Let $\g$ be a simply laced simple Lie algebra ($\l=1$), or $\g=osp(1|2n)$
($\l=\frac12$). Then for every $t\in\C^\times$:
\begin{enumerate}
\item $\Psi_t:W^k(\g)\to V^t(\h)\otimes F(\g_{1/2})$ is injective;
\item $\theta_t\bigl(\Psi_{t'}(W^{k'}(\g))\bigr)=\Psi_t\bigl(W^k(\g)\bigr)$,
where $k=t-h^\vee$ and $k'=t'-h^\vee$;
\item $\Psi_t^{-1}\circ\theta_t\circ\Psi_{t'}:W^{k'}(\g)\to W^k(\g)$ is an
isomorphism of conformal vertex algebras mapping $L$ to $L$. It therefore induces an
isomorphism $W_{k'}(\g)\cong W_k(\g)$ of the simple quotients.
\end{enumerate}
More generally, for every simple Lie algebra $\g$ the same argument gives
$W^k(\g)\cong W^{k'}(\g^\vee)$ for all non-critical $k$, with
$r^\vee(k+h^\vee)(k'+{}^Lh^\vee)=1$, where $r^\vee$ is the lacety of $\g$.
\end{proposition}

\begin{proof} Since $\g_0=\h$, one has $\kappa_{\g_0}|_\h=0$ and, $\kappa_\g$ being
$2h^\vee(\cdot|\cdot)$, the invariant bilinear form $\tau_k$ of \cite[(2.2)]{G} restricts on $\h$ to
$(k+h^\vee)(\cdot|\cdot)=t(\cdot|\cdot)$. Hence $V^{\tau_k}(\g_0)=V^t(\h)$, and
$\Psi_t$ is the Miura map
$\mu:W^k(\g)\to V^{\tau_k}(\g_0)\otimes F(\g_{1/2})$ of \cite[\S5]{G}, \cite[Corollary 3.1]{KWfree}
\cite[\S2.2]{N}. It is injective for every $k\in\C$. For simple Lie algebras this
is due to Frenkel and Arakawa (see \cite{Ar} and the references therein); for basic Lie superalgebras, an arbitrary
even nilpotent and an arbitrary good grading, including half-integral ones with
$\g_{1/2}\ne0$ as for $osp(1|2n)$, it is \cite[Thm.~14]{N}. This proves (1).

We now discuss the duality at generic levels, i.e. the following statement: there is a set $S\subset\C^\times$
with countable complement such that (2) holds for $t\in S$. For $\g$ simply
laced this is the generic Feigin--Frenkel duality \cite{FF}, which, in the
notation of Remark \ref{s:remscreen}, reads as follows: for generic $t$ and every
simple root $\a$, the screening operator $\mathcal Q_\a$ and the dual screening operator
$\mathcal Q^\vee_\a$ have the same kernel in $V^t(\h)$ (this is a rank one statement: both
kernels are the tensor product of the Heisenberg vertex algebra of $\a^\perp$ and
of the Virasoro vertex algebra of $\C\a$ for which $S_\a$ and $S^\vee_\a$ have
conformal weight $1$). Hence, if $t$ and $t'$ are generic, \eqref{s:screen} at
the level $t'$, \eqref{s:thetaQ}, and \eqref{s:screen} at the level $t$ give
$$
\theta_t\bigl(\Psi_{t'}(W^{k'}(\g))\bigr)
=\bigcap_{\a\in\Pi}\operatorname{Ker}\mathcal Q^{\vee,(t)}_\a
=\bigcap_{\a\in\Pi}\operatorname{Ker}\mathcal Q^{(t)}_\a=\Psi_t\bigl(W^k(\g)\bigr).
$$
For $\g=osp(1|2n)$ it is \cite[Lemma~6.3, Thm.~6.4,
Rem.~6.5]{G}. For $k$ generic in the sense of \cite[Def.~4.5]{G} (a set with
countable complement by \cite[Lemma~4.4]{G}), $\Psi_t(W^k(\g))$ is identified, via
$V^t(\h)\cong H$, $u\mapsto\sqrt t\,\hat u$, with the Fateev--Lukyanov algebra
$\mathcal WB_n(\gamma)\subset H\otimes F$. This algebra is generated by an odd
field $G_\gamma$ depending only on $\gamma=\sqrt{2t}-\frac1{\sqrt{2t}}$.
Choosing $\sqrt{2t'}=-1/\sqrt{2t}$ leaves $\gamma$ unchanged. The composite
$V^{t'}(\h)\cong H\cong V^t(\h)$ is then $u\mapsto\frac{\sqrt{t'}}{\sqrt t}u
=-\frac1{2t}u=\theta_t(u)$. So (2) holds whenever $k$ and $k'$ are both generic; by \eqref{s:screen}
and \eqref{s:thetaQ} this amounts to \eqref{s:dualscreen}, the two families of
screening operators being the two screening systems of the Fateev--Lukyanov
algebra \cite{FL}.
The set $S$ of such $t$ is the intersection of two sets with countable complement,
since $\sigma$ is a bijection of $\C^\times$; hence it has countable complement.

To upgrade the result from the generic case to all non-critical  values of $k$ we need to introduce an algebraic family of vertex algebras. Let $y=1/t$. By \cite[\S5]{G} the
$W$-algebras form a family $W_y(\g)$ over $\C[y]$. It is a free $\C[y]$-module,
freely strongly generated by fields of fixed conformal weights, with
$W^k(\g)=W_y(\g)\otimes_{\C[y]}\C_{1/t}$ for $t\ne0$. The Miura map is the
specialization of a $\C[y]$-linear map $\mu_y$ \cite[(5.2), Lemma~5.5]{G}. The
latter is written in terms of the rescaled currents $\bar J^u=y\,J^u$. Passing to
$J^u$ introduces only powers of $y^{\pm1}$. Let $\mathfrak V$ be the vector space
underlying $V^t(\h)\otimes F(\g_{1/2})$, spanned by the PBW monomials in the modes
$u_{(-n)}$, $u\in\h$, and $\Phi_{(-m)}$. As a graded vector space it does not
depend on $t$, and its conformal weight spaces $\mathfrak V_\Delta$ are
finite-dimensional. Fix $\Delta$, put $N=\dim W^k(\g)_\Delta$ (independent of $k$)
and $D=\dim\mathfrak V_\Delta$, and choose a basis of the free
$\C[t^{\pm1}]$-module $W_y(\g)_\Delta\otimes\C[t^{\pm1}]$. In this basis
$\Psi_t|_{W^k(\g)_\Delta}$ is a $D\times N$ matrix $M_\Delta(t)$ with entries in
$\C[t^{\pm1}]$. By (1),   it has rank $N$ for every $t\in\C^\times$. Hence its
$N\times N$ minors never vanish simultaneously, and
\[
\mathcal U_\Delta:\C^\times\to\Gr(N,\mathfrak V_\Delta),\qquad
t\mapsto\Psi_t\bigl(W^k(\g)\bigr)_\Delta ,
\]
is a morphism of varieties.
 On $\mathfrak V$, $\theta_t$ is the linear operator
multiplying a PBW monomial of degree $d$ in the $\h$-modes by $(-\l/t)^d$. Since
$\sigma$ is an automorphism of $\C^\times$,
$t\mapsto\theta_t\bigl(\mathcal U_\Delta(\sigma(t))\bigr)$ is again a morphism
$\C^\times\to\Gr(N,\mathfrak V_\Delta)$. By the generic case,  the two morphisms agree on
$S$, which is infinite. As $\C^\times$ is an irreducible curve and $\Gr(N,D)$ is
separated, they agree everywhere. This holds for every $\Delta$, which proves (2).

 By (1) and (2), $\Psi_t^{-1}\circ\theta_t\circ\Psi_{t'}$
is a well defined isomorphism of vertex algebras. It maps $L$ to $L$. Indeed,
write $\tilde\rho_t=t\rho^\vee-\rho=(t-\l)\rho^\vee$ and use $tt'=\l^2$. The
Sugawara part is preserved because $(\l/t)^2/t'=1/t$. For the derivative term,
\[
\theta_t\Bigl(\tfrac1{t'}T\tilde\rho_{t'}\Bigr)
=-\tfrac{\l}{tt'}(t'-\l)\,T\rho^\vee
=\bigl(1-\tfrac{\l}{t}\bigr)T\rho^\vee=\tfrac1t\,T\tilde\rho_t ,
\]
and $L^F$ is untouched, so $\theta_t(\Psi_{t'}(L))=\Psi_t(L)$ by \eqref{s:FFR}.
An isomorphism of the universal $W$-algebras maps the maximal ideal onto the maximal
ideal, and hence descends to the simple quotients.

For a simple Lie algebra $\g$, identify the Cartan subalgebra of $\g^\vee$ with $\h$ so that the roots of
$\g^\vee$ correspond to the coroots of $\g$. Then the normalized invariant bilinear form of $\g^\vee$
is $\frac1{r^\vee}$ times the induced one. The above proof scheme applies verbatim with
$\theta_t(u)=-u/(r^\vee t)$ and $t'=1/(r^\vee t)$, using  the generic
Feigin--Frenkel duality \cite{FF}, proving the last assertion. \end{proof}

\begin{remark}\label{rem:FFlevels} The $osp(1|2n)$ duality is available in the
literature for the following ranges of levels. \cite[Thm.~4.1]{CL} is an
isomorphism of one-parameter vertex algebras, i.e.\ it holds for generic levels.
The lift from the orbifolds in \cite[Rem.~4.1]{CL} rests on
\cite[Thm.~6.3]{CL}, whose proof in this case is left to the reader. The results
of \cite{CLNS} are stated for $k\notin\Q$. For $m=0$ the relative
semi-infinite cohomology of \cite{CLNS} is trivial, so the input they need is
again the orbifold duality together with the uniqueness of the extension.
Genra's proof \cite[Thm.~6.4]{G} is complete at generic levels, and its
extension to all levels requires the injectivity of the Miura map, which for Lie
superalgebras is \cite[Thm.~14]{N}. Proposition \ref{prop:FFall} covers all
non-critical levels uniformly.
\end{remark}

\endgroup

\begin{remark}\label{s:remGN} Gepner and Noyvert \cite{GN} studied unitarity of
the superconformal algebra $SW(\frac32,2)$. According to \cite{RD}, this conformal vertex algebra is 
isomorphic to the principal $W$-algebra $W^k(osp(3|2))$. In this case $c$
depends linearly on $k$, namely $c=18k+6$. According to \cite{GN}, this conformal
vertex algebra is unitary iff
$$
c=6+\frac{18}{p}\qquad\text{or}\qquad c=6-\frac{18}{p+1},\qquad p=3,4,\dots,\text{or $c=6$}
$$
Thus the values of $c$
for which the principal $W$-algebra of $osp(3|2)$ is unitary form an infinite
set,  accumulating at  $c=6$.
\end{remark}

\section{A generic non-unitarity result}

Let $\g$ be a basic Lie superalgebra, possibly a simple Lie algebra, and let
$\mathfrak s=span\{e,x,f\}\subseteq\g_{\ov 0}$ be an $sl_2$-subalgebra.  As in Section \ref{sec:super}, $f$, or $\mathfrak s$, is called
even if the eigenvalues of $2\ad x$ on $\g_{\ov 0}$ (resp.\ $\g_{\ov 1}$) are
even (resp.\ odd) integers; for $\g$ a Lie algebra this is the usual notion.
Denote by $W_k(\g,\mathfrak s)$ the simple $W$-algebra attached to $\g, \mathfrak s$ and $k\in \C$. Let $\g^\natural$ be the centralizer of $\mathfrak s$ in $\g$. No normalization of $(\cdot|\cdot)$ has to be fixed in this section.

\begin{remark} \label{lem:even}
The following facts are easy to prove. \par (1). 
$\mathfrak s$ is even if and only if all the strong generators $J^{\{v\}}$,
$v\in\g^f$ homogeneous, of $W^k(\g,\mathfrak s)$ are compatible, i.e.\ have
integral conformal weight if $v$ is even and half-integral conformal weight if $v$
is odd. \par
(2). If $\mathfrak s$ is even, then $\g_0\subseteq\g_{\ov 0}$ and
$\g_{1/2}\subseteq\g_{\ov 1}$; in particular $\g_0$ and $\g^\natural$ are
reductive Lie algebras.

\end{remark}

\begin{theorem}\label{prop:gen} (a).  Let $\g$ be  a simple Lie algebra, let the $sl_2$-subalgebra $\mathfrak s$ be even and suppose that
$\g^\natural$ is not abelian. Then $W_k(\g,\mathfrak s)$ can be unitary for at most a finite
number of $k$.\par 
(b).The same conclusion holds when $\g$ is a basic Lie superalgebra, provided the following condition holds: $(x|x)(\a|\a)>0$ for a root $\a$ of $\g^\natural$.\end{theorem}

\begin{proof} Suppose that  $W_k(\g,\mathfrak s)$ is unitary. Since $\g^\natural$ is not abelian, $W_k(\g,\mathfrak s)$ has an affine subalgebra with bracket
$$[J^{\{v\}\lambda} J^{\{v'\}}] = J^{\{[v,v']\}} + \lambda \left( k(v|v') + \frac{1}{2}\big(k_{\mathfrak{g}}(v,v') - k_{\mathfrak{g}_0}(v,v') - k_{1/2}(v,v')\big) \right),$$
provided that $v,v'\in \g^\natural$, where $k_{1/2}$ is the supertrace form of $\g_0$ on $\g_{1/2}$ (cf. \cite{KRW}).

Consider a simple component $\mathfrak a$ of $\g^\natural$ and let
$(\cdot|\cdot)_{\mathfrak a}$ be the restriction to $\mathfrak a$ of the invariant bilinear
form of $\g$. The level of the affine subalgebra corresponding to $\mathfrak a$  is  $k_{\mathfrak a}=ak+b$, 
and the sign of $a$ is the sign of $(\a|\a)$ (the normalized form of $\mathfrak a$ being positive definite on the real span
of its roots); in particular $a\ne0$.
Let $\phi$ be the conjugate linear involution with respect to which
$W_k(\g,\mathfrak s)$ is unitary. The space of vectors of conformal weight $1$ of
$W_k(\g,\mathfrak s)$ is a quotient of $\g^\natural$, on which $\phi$ acts as a
conjugate linear involution of Lie algebras. If the image of $\mathfrak a$ in this
quotient is zero, then $k_{\mathfrak a}=0$. If it is non-zero and it is not
$\phi$-stable, then, by \eqref{eq:top},
$(J^{\{v\}},J^{\{v\}})=-\langle\phi(J^{\{v\}})_{(1)}J^{\{v\}}\rangle=0$ for $v\in\mathfrak a$,
distinct simple components of $\g^\natural$ being orthogonal with respect to the
invariant form defined by the above $\lambda$-bracket; this is not possible, the
form of $W_k(\g,\mathfrak s)$ being positive definite. Otherwise the fields
$J^{\{v\}}$, $v\in\mathfrak a$, generate a $\phi$-stable quotient of
$V^{k_{\mathfrak a}}(\mathfrak a)$ on which the invariant form is positive
definite, and $k_{\mathfrak a}$ is a non-negative integer by
\cite[Thm.~11.7]{VB} (see \cite[Section 5.3]{KMP}).

Thus, since $k_{\mathfrak a}=ak+b$ is a non-negative integer and $a\ne 0$, if
$W_k(\g,\mathfrak s)$ is unitary, we have
\begin{equation}\label{c1}k\in \frac{1}{a}(\mathbb Z-b)\quad\text{ and }\quad
k\ge -\frac{b}{a}\ \text{ if } a>0,\qquad k\le -\frac{b}{a}\ \text{ if } a<0.\end{equation}
Observe now that unitarity implies $||L||^2=\langle L_{(3)}L\rangle=c(\mathfrak{g},\mathfrak s,k)/2\ge0$.
Recalling  from \cite[Remark 2.2]{KRW} the expression of the central charge of
$W_k(\g,\mathfrak s)$, and that $\g_0$ is even and $\g_{1/2}$ is odd by Remark
\ref{lem:even}, we obtain, for suitable coefficients $\alpha, \beta$,  the inequality
 \begin{align}\notag
c(\mathfrak{g},\mathfrak s,k) &= \operatorname{dim} \mathfrak{g}_0 {}+\tfrac12\operatorname{dim}\g_{1/2} -\frac{12}{k+h^\vee}\| \rho - (k + h^{\vee})x \|^2\\
&=\frac{\alpha + \beta k  - 12(x|x) k^2}{k+h^\vee}=:\frac{Q(k)}{k+h^\vee}\ \ge0\label{c2}.
\end{align}
 Put $s=(x|x)$, so that $as>0$ by
assumption.  If $s>0$ and $a>0$, then on $k>-h^\vee$ the inequality \eqref{c2}
reads $Q(k)\ge0$ and $Q$ is concave, so $k$ is bounded above, while
$-b/a\le k\le -h^\vee$ is a bounded interval.  If $s<0$ and $a<0$, then on
$k<-h^\vee$ the inequality \eqref{c2} reads $Q(k)\le0$ and $Q$ is convex, so $k$
is bounded below, while $-h^\vee\le k\le -b/a$ is a bounded interval.  In both
cases conditions \eqref{c1} and \eqref{c2} describe a discrete and bounded, hence finite, set of real numbers.
\end{proof}

For the classical series the hypotheses of Theorem \ref{prop:gen} can be read
off the block decomposition of the defining module.  Let $\g=sl(m|n)$, $m\ne n$,
or $\g=osp(M|2N)$, let $f\in\g_{\ov 0}$ be a nilpotent element, and let
$V=\C^{m|n}$, resp.\ $\C^{M|2N}$, be the defining module.  Since the
$sl_2$-triple lies in $\g_{\ov 0}$,
$$
V=\bigoplus_{\ell\ge1}M_\ell\otimes U_\ell,\qquad U_\ell=\C^{p_\ell|q_\ell},
$$
where $M_\ell$ is the $\ell$-dimensional irreducible $sl_2$-module, a purely even
space, and $U_\ell$ is the multiplicity superspace.  All the invariant bilinear forms below are computed
with respect to the supertrace form of $V$.

\begin{lemma}\label{lem:blockeven} The nilpotent element $f$ is even if and only if the
sizes $\ell$ of the blocks contained in $V_{\ov 0}$, i.e.\ those with $p_\ell>0$,
all have the same parity, and the sizes of the blocks contained in $V_{\ov 1}$,
i.e.\ those with $q_\ell>0$, all have the other parity.  In particular, if $f$ is
even then every $U_\ell$ is either purely even or purely odd.
\end{lemma}
\begin{proof} The eigenvalues of $2x$ on $M_\ell$ are congruent to $\ell-1$
modulo $2$.  The eigenvalues of $2\ad x$ on $\g_{\ov 0}$ are the differences of
two eigenvalues of $2x$ on $V_{\ov 0}$, or on $V_{\ov 1}$, and all of them occur,
while those on $\g_{\ov 1}$ are the differences of an eigenvalue on $V_{\ov 0}$ and
one on $V_{\ov 1}$; for $\g=osp(M|2N)$ some differences have to be replaced by
sums, which does not affect parities.  The claim follows.
\end{proof}

Note that, if $f$ is even, then $\g^\natural$  is a Lie algebra, hence every simple component $\mathfrak a$ of $\g^\natural$ has a unique bilinear invariant normalized form $(\cdot |\cdot)_{\mathfrak a}$, and one defines the {\sl index} of $\mathfrak a$ in $\g$ as the ratio
$\frac{(\cdot |\cdot)_{|\mathfrak a}}{(\cdot |\cdot)_{\mathfrak a}}$.
\begin{prop}\label{prop:blocks} Let $f$ be even. With the notation above:
\begin{enumerate}
\item $\displaystyle 12\,(x|x)=\sum_{\ell\ge1}(p_\ell-q_\ell)(\ell^3-\ell)$;
\item $\g^\natural=\bigoplus_\ell\g^\natural_\ell$ with
$\g^\natural_\ell\subseteq gl(U_\ell)$, and the restriction of the invariant form
of $\g$ to $\g^\natural_\ell$ is $\ell\cdot str_{U_\ell}$.  Hence every simple
component of $\g^\natural_\ell$ has index of sign $+1$ if $U_\ell$ is even, and
of sign $-1$ if $U_\ell$ is odd;
\item consequently the hypotheses of Theorem \ref{prop:gen} hold if and only if
for some $\ell$ the Lie algebra $\g^\natural_\ell$ is not abelian and
either $U_\ell$ is even and $(x|x)>0$, or $U_\ell$ is odd and $(x|x)<0$.
\end{enumerate}
Here $\g^\natural_\ell$ is $sl(U_\ell)$, up to centre, if $\g=sl(m|n)$; if
$\g=osp(M|2N)$ it is $so(U_\ell)$ when $\ell$ is odd and $U_\ell$ even, or $\ell$
is even and $U_\ell$ odd, and $sp(U_\ell)$ otherwise.  It is non-abelian if and
only if $\dim U_\ell\ge2$, resp.\ $\dim U_\ell\ge3$ in the orthogonal case.  By
Lemma \ref{lem:blockeven}, in the orthosymplectic case the components are either
all orthogonal or all symplectic.
\end{prop}

\begin{proof} (1) The element $x$ acts on $M_\ell$ with eigenvalues
$\frac{\ell-1}2,\frac{\ell-3}2,\dots,-\frac{\ell-1}2$, so that
$tr_{M_\ell}(x^2)=\frac{\ell^3-\ell}{12}$, and $x$ acts on $M_\ell\otimes U_\ell$
as $x\otimes1$, whence
$str_V(x^2)=\sum_\ell (p_\ell-q_\ell)\,tr_{M_\ell}(x^2)$.

(2) The centralizer of the triple in $gl(V)$ is $\bigoplus_\ell gl(U_\ell)$,
acting as $1\otimes A$, and $\g^\natural$ is its intersection with $\g$.  As
$M_\ell$ is purely even,
$str_V\bigl((1\otimes A)(1\otimes B)\bigr)=\ell\,str_{U_\ell}(AB)$, while
operators supported on different $\ell$ are orthogonal.  By Lemma
\ref{lem:blockeven}, $U_\ell$ is purely even or purely odd, so that
$\g^\natural_\ell$ is a Lie algebra.  If $U_\ell$ is even, the trace form of the
defining module is a positive multiple of the normalized form on a simple
subalgebra of $gl(U_\ell)$, so the index is a positive multiple of $\ell$;
if $U_\ell$ is odd one has $str=-tr$, and the index is a negative multiple
of $\ell$.  The description of $\g^\natural_\ell$ in the orthosymplectic case
follows from the fact that the invariant form of $M_\ell$ is symmetric for $\ell$
odd and skew for $\ell$ even, while the form of $V$ is symmetric on $V_{\ov 0}$
and skew on $V_{\ov 1}$.

(3) is now Theorem \ref{prop:gen}, the sign of the index being the sign of
$(\a|\a)$.
\end{proof}

Proposition \ref{prop:blocks} translates into a condition on two partitions.
Denote by $\l$ the partition given by the sizes of the blocks of $f$ on
$V_{\ov 0}$ and by $\mu$ the one on $V_{\ov 1}$, so that $p_\ell$, resp.\
$q_\ell$, is the multiplicity of $\ell$ in $\l$, resp.\ in $\mu$, and set
$$
S(\nu)=\sum_{\ell\in\nu}(\ell^3-\ell)
$$
(the sum over the parts of $\nu$, with multiplicity), so that
$12(x|x)=S(\l)-S(\mu)$ by Proposition \ref{prop:blocks}(1).  Call \emph{heavy}
whichever of $V_{\ov 0}$, $V_{\ov 1}$ has the larger $S$: by Proposition
\ref{prop:blocks}(2), a simple component of $\g^\natural$ attached to the heavy
side has index of the same sign as $(x|x)$, and one attached to the other side
has the opposite sign.

\begin{corollary}\label{cor:blocks} Let $\g=sl(m|n)$, $m\ne n$, or
$\g=osp(M|2N)$, and let $f\in\g_{\ov 0}$ be even, with partitions $\l$ and $\mu$
as above.  Then the hypotheses of Theorem \ref{prop:gen} hold if and only if
$S(\l)\ne S(\mu)$ and, in addition,
\begin{enumerate}
\item for $\g=sl(m|n)$: the partition of the heavy side has a part of
multiplicity $\ge2$;
\item for $\g=osp(M|2N)$ with the parts of $\l$ even and those of $\mu$ odd: no further condition;
\item for $\g=osp(M|2N)$ with the parts of $\l$ odd and those of $\mu$ even: the
partition of the heavy side has a part of multiplicity $\ge3$.
\end{enumerate}
\end{corollary}

\begin{proof} By Proposition \ref{prop:blocks}(3) the hypotheses hold if and only
if $(x|x)\ne0$, i.e.\ $S(\l)\ne S(\mu)$, and some $\g^\natural_\ell$ attached to
the heavy side is not abelian, which means $r\ge2$ for a component $sl(r)$ or
$sp(r)$ and $r\ge3$ for a component $so(r)$.  In case (1) all components are of
type $sl$.  In case (3) all of them are orthogonal, and in case (2) all of them
are symplectic, by Proposition \ref{prop:blocks}.  Moreover in case (2) every
multiplicity is even: $\l$ is the partition of a nilpotent element of $so(M)$,
whose even parts have even multiplicity, and all its parts are even, while $\mu$
is the partition of a nilpotent element of $sp(2N)$, whose odd parts have even
multiplicity, and all its parts are odd.  Hence $p_\ell,q_\ell\ge2$ whenever they
do not vanish, and each side carries a non-abelian component.
\end{proof}

\begingroup
\subsection{Conformal embedding and unitarity when $\g^\natural$ is semisimple}\label{sub:ss}
We first record two facts which hold in general.

\begin{lemma}\label{lem:knat0}
Let $\mathfrak s$ be an even $sl_2$ subalgebra of a basic Lie superalgebra $\g$, and assume that the level $k^\natural_{\mathfrak a}$ of a
simple ideal $\mathfrak a$ of $\g^\natural$ vanishes. Then every vector of
$W^k(\g,\mathfrak s)$ lying in an $\mathfrak a$-isotypic component of non-trivial
type belongs to the maximal ideal. In particular, if $\g^\natural$ is semisimple and all its simple ideals have
level $0$ (e.g.\ $\g^\natural$ is simple and $k^\natural=0$), then
$W_k(\g,\mathfrak s)$ is strongly generated by the images of the
$\g^\natural$-invariant strong generators of $W^k(\g,\mathfrak s)$.
\end{lemma}
\begin{proof}
Let $(\cdot,\cdot)$ be the invariant symmetric bilinear form on
$W^k(\g,\mathfrak s)$ normalized by $(\one,\one)=1$; it exists and is unique,
since $W^k(\g,\mathfrak s)_0=\C\one$ and $L_1W^k(\g,\mathfrak s)_1=0$
(cf.\ \cite{Li}), and its radical is the maximal ideal
$I$. The space $W^k(\g,\mathfrak s)_1$ is spanned by the $J^{\{v\}}$,
$v\in\g^\natural$, and distinct $L_0$-eigenspaces are orthogonal. For
$a\in\mathfrak a$ one has $(J^{\{a\}},J^{\{v\}})=\pm k^\natural_{\mathfrak a}(a|v)_{\mathfrak a}=0$
if $v\in\mathfrak a$, and $(J^{\{a\}},J^{\{v\}})=0$ if $v$ lies in another ideal
of $\g^\natural$, distinct ideals being orthogonal for every invariant bilinear
form. Hence $J^{\{a\}}$ lies in the radical, i.e.\ in $I$. A vector $u$ in a
non-trivial isotypic component is a sum of vectors $J^{\{a\}}_{(0)}u'$ with
$a\in\mathfrak a$, and $J^{\{a\}}_{(0)}u'\in I$ because $I$ is an ideal. For the
last statement, recall that $W^k(\g,\mathfrak s)$ is spanned by the normally
ordered monomials $:\!\partial^{n_1}A^1\cdots\partial^{n_r}A^r\!:$ in its strong
generators, which can be chosen in the isotypic components of $\g^\natural$;
every monomial containing a generator which is not $\g^\natural$-invariant lies
in $I$.
\end{proof}

\begin{lemma}\label{lem:coset}
Let $\mathfrak s$ be even and $\g^\natural$ semisimple, with simple ideals
$\mathfrak a_1,\dots,\mathfrak a_r$; let $k_i$ be the level of $\mathfrak a_i$
and $h^\vee_i$ its dual Coxeter number. If $W_k(\g,\mathfrak s)$ is unitary, then
$k_i\in\Z_{\ge0}$ for all $i$ and
\begin{equation}\label{cosetbound}
c(\g,\mathfrak s,k)\ \ge\ c_{\g^\natural}:=\sum_{i=1}^r\frac{k_i\dim\mathfrak a_i}{k_i+h^\vee_i}.
\end{equation}
More precisely, if $L^\natural$ is the Sugawara vector of the affine subalgebra,
then
\begin{equation}\label{cosetnorm}
(L-L^\natural,L-L^\natural)=\tfrac12\bigl(c(\g,\mathfrak s,k)-c_{\g^\natural}\bigr).
\end{equation}
If $\g^\natural$ is simple we write $k^\natural$, $h^\vee_\natural$ for $k_1$,
$h^\vee_1$.
\end{lemma}
\begin{proof}
By the proof of Theorem \ref{prop:gen}, $k_i\in\Z_{\ge0}$ for all $i$, and the
ideals $\mathfrak a_i$ with $k_i>0$ are stable under the involution; hence
$L^\natural$ is fixed by the involution. The fields $J^{\{a\}}$,
$a\in\g^\natural$, being primary of conformal weight $1$
\cite[Thm.~7.4(b)]{KMP}, $L-L^\natural$ is a quasiprimary vector commuting with
the affine subalgebra, and it is a Virasoro vector of central charge
$c(\g,\mathfrak s,k)-c_{\g^\natural}$. Hence
$(L-L^\natural)_{(3)}(L-L^\natural)=\frac12(c(\g,\mathfrak
s,k)-c_{\g^\natural})\one$, and \eqref{eq:top} gives \eqref{cosetnorm} (cf.\
the proof of \cite[Thm.~3.1]{AMP23}).
\end{proof}

If equality holds in \eqref{cosetbound}, then $L=L^\natural$ in
$W_k(\g,\mathfrak s)$, i.e.\ the affine subalgebra is conformally embedded. In
this case a unitary $W$-algebra reduces to its affine subalgebra.

\begin{proposition}\label{prop:conformal}
Let $\g$ be a simple Lie algebra, and let $\mathfrak s$ be even with
$\g^\natural$ semisimple. If $W_k(\g,\mathfrak s)$ is unitary and
$c(\g,\mathfrak s,k)=c_{\g^\natural}$, then $W_k(\g,\mathfrak s)$ coincides with
its affine subalgebra:
$$
W_k(\g,\mathfrak s)\cong V_{k_1}(\mathfrak a_1)\otimes\cdots\otimes V_{k_r}(\mathfrak a_r).
$$
\end{proposition}
\begin{proof}
Write $W=W_k(\g,\mathfrak s)$. By Lemma \ref{lem:coset}, $k_i\in\Z_{\ge0}$ and
$L-L^\natural$ has norm zero, hence $L=L^\natural$ in $W$. Moreover the affine
subalgebra of $W$ is simple, the invariant form being positive definite on it,
hence it is isomorphic to $\bigotimes_iV_{k_i}(\mathfrak a_i)$. We now argue as
in \cite[Thm.~3.2]{AMP23}. Recall that $W^k(\g,\mathfrak s)$ is strongly
generated by fields $J^{\{v\}}$, $v\in\g^f$, of conformal weight $1-j$ if
$[x,v]=jv$, which can be chosen in such a way that
$J^{\{a\}}_{(0)}J^{\{v\}}=J^{\{[a,v]\}}$ for $a\in\g^\natural$
\cite[Thm.~4.1]{AMP23}. Let $R$ be an irreducible constituent of the
$\g^\natural$-module $\{v\in\g^f\mid[x,v]=(1-\Delta_R)v\}$, $\Delta_R\ge2$, with
highest weight $\l_i$ with respect to $\mathfrak a_i$, and set
\begin{equation}\label{hR}
h_R=\sum_{i=1}^r\frac{(\l_i|\l_i+2\rho_i)_i}{2(k_i+h^\vee_i)}.
\end{equation}
By \cite[Thm.~3.2]{AMP23}, whose proof applies verbatim to any quotient of
$W^k(\g,\mathfrak s)$ in which $L=L^\natural$, the vertex algebra $W$ is
strongly generated by $\g^\natural$ and by the images of the constituents $R$
in a minimal set $K$, and $h_R=\Delta_R$ for all $R\in K$. Hence $W$ coincides
with its affine subalgebra if $h_R\ne\Delta_R$ for all $R$. Our setting differs
from that of \cite{AMP23} (and of \cite[Criterion 2.4]{AACLMMP}) in two
respects.
\begin{enumerate}
\item In \cite[Thm.~3.1]{AMP23} the equality $L=L^\natural$ is derived from
$c(\g,\mathfrak s,k)=c_{\g^\natural}$ under a hypothesis on the strong
generators of conformal weight $2$, which holds if $L$ is, up to multiples, the
only $\g^\natural$-invariant one (condition (D) of
\cite[Criterion 2.4]{AACLMMP}). For $\g$ of classical type the latter condition
holds only for the nilpotent elements of Jordan type $(q^l,1^r)$
\cite[Prop.~4.1, 4.9]{AACLMMP}. Here $L=L^\natural$ follows from unitarity, and
no hypothesis on the generators is needed.
\item In \cite{AMP23}, \cite{AACLMMP} the condition $h_R\ne\Delta_R$ is required
for all $R$. Here it is needed only for the constituents which are integrable.
Indeed, let $R\in K$ and let $W'$ be the vertex subalgebra of $W$ generated by
the vectors of conformal weight less than $\Delta_R$. In the proof of
\cite[Thm.~3.2]{AMP23} it is shown that the image of $R$ in $W/W'$ is non-zero
and is annihilated by the operators $J^{\{a\}}_{(n)}$, $a\in\g^\natural$,
$n\ge1$. Since $W$ is unitary, $W/W'$ is isomorphic to the orthogonal complement
of $W'$, which is a unitary module over the affine Lie algebra of
$\g^\natural$. Therefore $R$ generates a unitary highest weight module, which
is integrable \cite[Thm.~11.7]{VB}:
\begin{equation}\label{integr}
(\l_i|\theta_i^\vee)\le k_i,\qquad i=1,\dots,r,
\end{equation}
where $\theta_i$ is the highest root of $\mathfrak a_i$.
\end{enumerate}
We are thus reduced to proving that $h_R\ne\Delta_R$ for the constituents $R$
which satisfy \eqref{integr}.

Let $\g$ be of classical type. We use the notation of Proposition
\ref{prop:blocks}, with all $U_\ell$ even. Since $\g$ is $sl(V)$, $S^2V$ or
$\Lambda^2V$ and $V=\bigoplus_\ell M_\ell\otimes U_\ell$, every non-trivial
irreducible constituent of the $\g^\natural$-module $\g$ is of one of the
following types: (a) the adjoint representation of a simple ideal; (b) the
non-trivial constituent $S^2_0U_\ell$ of $S^2U_\ell$ if
$\g^\natural_\ell=so(U_\ell)$, or $\Lambda^2_0U_\ell$ of $\Lambda^2U_\ell$ if
$\g^\natural_\ell=sp(U_\ell)$; (c) $U_\ell\otimes U_{\ell'}$, $\ell\ne\ell'$.
If $\g=sl(V)$ only (a) occurs, $\g^\natural$ being semisimple. Let
$m=\dim U_\ell$, and let $\kappa$ be the level of $\g^\natural_\ell$ with
respect to the form $tr_{U_\ell}$ if $\g^\natural_\ell$ is $sl(U_\ell)$ or
$sp(U_\ell)$, and to $\frac12tr_{U_\ell}$ if it is $so(U_\ell)$, $m\ge3$; by
\eqref{integr} we may assume $\kappa>0$. The contributions to $h_R$ of $U_\ell$,
of the adjoint representation and of the representation in (b) are
$$
\begin{array}{llll}
sl(U_\ell):& &\frac m{\kappa+m};&\\[4pt]
sp(U_\ell):&\frac{m+1}{4(\kappa+\frac m2+1)},&\frac{\frac m2+1}{\kappa+\frac m2+1},&\frac{\frac m2}{\kappa+\frac m2+1};\\[4pt]
so(U_\ell):&\frac{m-1}{2(\kappa+m-2)},&\frac{m-2}{\kappa+m-2},&\frac{m}{\kappa+m-2}.
\end{array}
$$
Here $\kappa\ge1$, except for $so(3)$, where $\kappa\in\frac12\Z_{>0}$ and
\eqref{integr} requires $\kappa\ge1$ for $U_\ell$ and $\kappa\ge2$ for
$S^2_0U_\ell$. Hence the contribution of $U_\ell$ is at most $\frac12$ and the
other ones are at most $\frac43$, so that $h_R\le\frac43<2\le\Delta_R$. Note
that integrability cannot be dropped: for $\g^\natural_\ell=so(3)$,
$\kappa=\frac12$ and $R=S^2_0U_\ell$ one has $h_R=2$.

Let now $\g$ be of exceptional type. By inspection of
\cite[Tables 2--14]{AACLMMP}, for $\mathfrak s$ even and $\g^\natural$ semisimple the
conditions $k_i\in\Z_{\ge0}$ for all $i$ and
$c(\g,\mathfrak s,k)=c_{\g^\natural}$ hold exactly in the following cases,
where $\mathfrak s$ is denoted by the Bala--Carter label of $f$:
$$
\begin{array}{ll}
F_4:& A_2\ (k=-\tfrac92),\quad B_3\ (k=-\tfrac{59}8);\\[3pt]
E_6:& A_2\ (k=-6);\\[3pt]
E_7:& A_2+3A_1\ (k=-\tfrac{21}2),\quad D_4(a_1)\ (k=-12),\quad
A_3+A_2+A_1\ (k=-\tfrac{40}3,\,-\tfrac{53}4,\,-\tfrac{25}2),\\[3pt]
& A_4+A_2\ (k=-\tfrac{71}5,\,-\tfrac{41}3),\quad
A_6\ (k=-\tfrac{108}7,\,-\tfrac{107}7);\\[3pt]
E_8:& 2A_2\ (k=-18),\quad D_4(a_1)+A_2\ (k=-\tfrac{65}3,\,-21).
\end{array}
$$
In \cite{AACLMMP}  all these levels, with the exception of $E_7$, $D_4(a_1)$, $k=-12$,
are proved to be collapsing, i.e.\ at these levels $W_k(\g,\mathfrak s)$
coincides with its affine subalgebra, which is our claim. In the remaining case
$c(\g,\mathfrak s,k)=0$; since $W$ is unitary, $(L,L)=\frac c2=0$, hence $L=0$
and $W=W_0=\C$. This is again our claim, because $c_{\g^\natural}=c(\g,\mathfrak s,k)=0$
forces $k_i=0$ for all $i$.
\end{proof}

\subsection{Rectangular nilpotents}\label{sub:rect} Lemma \ref{lem:coset} sharpens the inequality \eqref{c2} used in the proof of
Theorem \ref{prop:gen}. For rectangular nilpotents it determines all the levels
at which unitarity is possible. Let $q,m\ge2$, $N=qm$, and let $f=f_{q^m}\in sl_N$
be a nilpotent element of Jordan type $q^m$. By \cite[Section 10]{AMP23},
$\g^\natural\cong sl_m$, $k^\natural=q(k+N-m)$, and $W^k(sl_N,f)$ is strongly
generated by the fields $J^{\{a\}}$, $a\in\g^\natural$, by $L$, by fields $W_i$,
$3\le i\le q$, of conformal weight $i$, on which $\g^\natural$ acts trivially,
and by fields of conformal weight $i$, $2\le i\le q$, which for each $i$ span a
copy of the adjoint representation of $\g^\natural$ \cite[Thm.~10.1]{AMP23}.
Moreover
$$
c(sl_N,f,k)=\frac{k\bigl(m(q-q^3)(k+N)+N^2-1\bigr)}{k+N}-m^2q(q^3-2q^2+1),
\qquad c_{\g^\natural}=\frac{(m^2-1)k^\natural}{k^\natural+m},
$$
and a direct computation gives
\begin{equation}\label{cosetrect}
c(sl_N,f,k)-c_{\g^\natural}=
-\frac{m(q-1)\,(k^\natural-1)(k^\natural+1)\bigl(N+(q+1)k^\natural\bigr)}
{(k^\natural+m)(k^\natural+N)} .
\end{equation}
The three zeros of the right hand side are the conformal levels of
\cite[(10.4)]{AMP23}.

\begin{proposition}\label{prop:rect}
Let $q,m\ge2$ and $N=qm$. If  $W_k(sl_N,f_{q^m})$ is unitary, then
$k^\natural\in\{0,1\}$, i.e.\ either $k=-N+m$, and then $c=q-1$, or
$k=-N+m+\frac1q$, and then $c=m-1$.
\end{proposition}
\begin{proof}
By the proof of Theorem \ref{prop:gen}, $k^\natural\in\Z_{\ge0}$. If
$k^\natural\ge2$ the right hand side of \eqref{cosetrect} is negative, against
Lemma \ref{lem:coset}. For $k^\natural=0$ one has $c_{\g^\natural}=0$ and
\eqref{cosetrect} gives $c=q-1$; for $k^\natural=1$ it gives
$c=c_{\g^\natural}=m-1$.
\end{proof}

At $k^\natural=1$ unitarity does hold, and the $W$-algebra collapses to its
affine subalgebra.

\begin{theorem}[{\cite[Thm.~10.2]{AMP23}}, see also \cite{AvEM}]\label{thm:qm}
Let $q,m\ge2$ and $N=qm$. Then
$$
W_{-N+m+\frac1q}(sl_N,f_{q^m})\cong V_1(sl_m).
$$
In particular this $W$-algebra is unitary.
\end{theorem}

At $k^\natural=0$, by Lemma \ref{lem:knat0}, $W_{-N+m}(sl_N,f_{q^m})$ is strongly
generated by the images of $L$ and of $W_3,\dots,W_q$, and has central charge
$q-1$. For $q=2$ this settles the question.

\begin{corollary}\label{cor:2mall}
For all $m\ge2$, $W_k(sl_{2m},f_{2^m})$ is unitary if and only if
$k\in\{-m,-m+\frac12\}$, and
$$
W_{-m}(sl_{2m},f_{2^m})=Vir_{c=1},\qquad W_{-m+\frac12}(sl_{2m},f_{2^m})\cong V_1(sl_m).
$$
\end{corollary}
\begin{proof}
By Proposition \ref{prop:rect} only these two levels are possible. At $k=-m$
the simple quotient is strongly generated by $L$ and has $c=1$, hence it is a
quotient of the universal Virasoro vertex algebra of central charge $1$, which
is simple and unitary. The level $k=-m+\frac12$ is Theorem \ref{thm:qm}.
\end{proof}

For $m=2$, $f_{2^2}$ is the short nilpotent element $f_{sh}$ of $sl_4$, and
Corollary \ref{cor:2mall} says that $W_k(sl_4,f_{sh})$ is unitary exactly for
$k=-2$ and $k=-\frac32$, with $W_{-2}(sl_4,f_{sh})=Vir_{c=1}$ and
$W_{-3/2}(sl_4,f_{sh})\cong V_1(sl_2)$; in this case the statement can also be
checked using the explicit $\lambda$-brackets of \cite{AMPenn}. The next example
shows that, for
$q\ge3$, invariant generators of higher conformal weight may survive at
$k^\natural=0$, and that they may produce unitary $W$-algebras which are not
generated by $L$ and $V^{k^\natural}(\g^\natural)$.

\begin{proposition}\label{prop:33}
Let $\g=sl_6$ and $f$ of Jordan type $3^2$, so that $\g^\natural\cong sl_2$,
$k^\natural=3(k+4)$ and $c=203-48t-\frac{210}t$, $t=k+6$.
\begin{enumerate}
\item $W_k(sl_6,f)$ is unitary if and only if $k\in\{-4,-\frac{11}3\}$.
\item $W_{-11/3}(sl_6,f)\cong V_1(sl_2)$.
\item $W_{-4}(sl_6,f)\cong W_{-2}(sl_3)$, the principal $W$-algebra of central
charge $2$. In particular $W_{-4}(sl_6,f)$ is unitary, but it is not generated by
$L$ and the image of $V^{k^\natural}(\g^\natural)$.
\end{enumerate}
\end{proposition}
\begin{proof}
By Proposition \ref{prop:rect} only the levels $k=-4$ and $k=-\frac{11}3$ are
possible, and (2) is Theorem \ref{thm:qm}. At $k=-4$ one has $k^\natural=0$ and
$c=2$, and by Lemma \ref{lem:knat0} $W_{-4}(sl_6,f)$ is strongly generated by the images of $L$ and of
the invariant generator $W=W_3$ of conformal weight $3$. If the image of $W$
lies in $\C\,\partial L$, then $W_{-4}(sl_6,f)=Vir_{c=2}$, which is unitary.
Otherwise, replacing $W$ by $W+a\,\partial L$ we may assume $(\partial L,W)=0$;
since the weight one space is zero, $L_1W\in\C L$ and $L_2W=0$, and
$(L,L_1W)=(L_{-1}L,W)=0$ gives $L_1W=0$, while $L_3W=(\partial L,W)\one=0$. Thus
$W$ is primary, and $(W,W)\ne0$ because $W$ is not in the radical. Hence
$W_{-4}(sl_6,f)$ is a simple vertex algebra strongly generated by a Virasoro
field of central charge $2$ and a primary field of conformal weight $3$ with
non-zero norm. By the uniqueness of the $W_3$ algebra \cite{Za}, \cite{Lin}, it
is the simple quotient of the universal $W_3$-algebra of central charge $2$,
i.e.\ $W_{k'}(sl_3)$ with $2-\frac{24(k'+2)^2}{k'+3}=2$, i.e.\ $k'=-2$, which is
unitary by Theorem \ref{4.2}, since $c=2=\rank(sl_3)$ (see also \cite{CTW}). In
both cases $W_{-4}(sl_6,f)$ is unitary, which proves (1). Finally, a computer
calculation of the radical of the invariant form in conformal weight $\le3$,
based on the Miura realization of \cite{AM}, shows that the homogeneous
components of conformal weight $1,2,3$ of $W_{-4}(sl_6,f)$ have dimensions
$0,1,2$; hence the image of $W$ is not in $\C\,\partial L$, and the second case
occurs. This proves (3).
\end{proof}

Note that $k=-4$ is not an admissible level for $sl_6$, so that
\cite[Thm.~1.1]{AvEM} is not available there.

\endgroup

\section{Unitarity of the $W$-algebras of $sl_4$}\label{sec:sl4}
As an illustration of the previous results, we collect in Table \ref{tab:sl4} what
is known on the unitarity of the simple $W$-algebras $W_k(sl_4,\mathfrak s)$,
for all the $sl_2$-subalgebras $\mathfrak s$ of $sl_4$. (The affine vertex algebra
$V_k(sl_4)$, which corresponds formally to $f=0$, is unitary exactly for $k\in\Z_{\ge0}$
\cite[Thm.~11.7]{VB}.) The conjugacy classes of the
subalgebras $\mathfrak s$ are parametrized by the partitions $\l\ne(1^4)$ of $4$, where
$\l$ is the Jordan type of $f$; $\mathfrak s$ is even for all $\l$ except
$\l=(2,1^2)$. We set $t=k+4$. Since $\|\rho\|^2=5$, the central charge is
\cite[Remark 2.2]{KRW}
$$
c(sl_4,\mathfrak s,k)=\dim\g_0-\tfrac12\dim\g_{1/2}+24(\rho|x)-\frac{60}t-12(x|x)\,t .
$$
We denote by $\mathcal H$ the Heisenberg vertex algebra of rank one with its standard
conformal vector, of central charge $1$, by $Vir_{c=1}$ the Virasoro vertex
algebra of central charge $1$, and by $V_{\Z\a}$ the lattice vertex algebra of
the lattice $\Z\a$.

\begin{table}[h]
\renewcommand{\arraystretch}{1.5}
\setlength{\tabcolsep}{4pt}
\footnotesize
\begin{tabular}{|c|c|c|l|c|l|l|}
\hline
$\l$ & $\g^\natural$ & $c(sl_4,\mathfrak s,k)$ & unitary for & $c$ & $W_k(sl_4,\mathfrak s)$ & source\\
\hline
$2,1^2$ & $gl_2$ & $39-\frac{60}t-6t$ & $k=-1$ & $1$ & $\mathcal H$ &
\cite[Thm.~7.9, Cor.~7.12]{KMP}\\
\hline
$2^2$ & $sl_2$ & $55-\frac{60}t-12t$ & $k=-2$ & $1$ & $Vir_{c=1}$ & Corollary \ref{cor:2mall}\\
 & & & $k=-\frac32$ & $1$ & $V_1(sl_2)$ & Corollary \ref{cor:2mall}, \cite{AMP23}\\
\hline
$3,1$ & $\C$ & $77-\frac{60}t-24t$ & $k=-\frac83$ & $0$ & $\C$ & Proposition \ref{prop:31}, \cite{AvEM}\\
 & & & $k=-\frac52$ & $1$ & $\mathcal H$ & Proposition \ref{prop:31}, \cite{AMP23}\\
 & & & $k=-\frac73$ & $1$ & $V_{\Z\a}$, $(\a|\a)=4$ & Proposition \ref{prop:31}, \cite{ACGHR}\\
\hline
$4$ & $0$ & $123-\frac{60}t-60t$ & $k<-4$ or $|k+4|=1$ & $c\ge3$ & --- & Theorem \ref{4.2}\\
 & & & $k+4\in\{\frac p{p+1},\frac{p+1}p\}$, $p\ge4$ & $3-\frac{60}{p(p+1)}$ &
 ($\C$ if $p=4$) & Remark \ref{rem:ACL}, \cite{ACL}\\
\hline
\end{tabular}
\medskip
\caption{Unitarity of $W_k(sl_4,\mathfrak s)$; here $t=k+4$.}\label{tab:sl4}
\end{table}

In the first three  rows of Table \ref{tab:sl4} the list is complete:
$W_k(sl_4,\mathfrak s)$ is unitary if and only if $k$ is one of the listed
values. In the last row it is expected to be complete (Remark \ref{rem:ACL}).\par
The only case in which unitarity is not already detected  is 
 $\l=(3,1)$. In the following we deal with this case. Here $f$ is a subregular nilpotent element. It is even, but
Theorem \ref{prop:gen} does not apply, since $\g^\natural$ is abelian.
Nevertheless the method of its proof, in the sharper form of Lemma
\ref{lem:coset}, still determines all the unitary levels: see Proposition
\ref{prop:31} below.

We may take $x=\mathrm{diag}(1,0,-1,0)$, and $\g^\natural=\C v$ with
$v=\mathrm{diag}(1,1,1,-3)$. The vertex algebra $W^k(sl_4,f)$ is strongly
generated by $J=J^{\{v\}}$, of conformal weight $1$, by $L$, by two fields
$G^\pm$ of conformal weight $2$ such that $J_{(0)}G^\pm=\pm4\,G^\pm$, and by a
field $W$ of conformal weight $3$ such that $J_{(0)}W=0$. In particular $J_{(0)}$
is diagonalizable with eigenvalues in $4\Z$, and the vectors of conformal weight
$2$ on which $J_{(0)}$ acts by $\pm4$ are the multiples of $G^\pm$. By the
formula for the $\l$-bracket given in the proof of Theorem \ref{prop:gen},
\begin{equation}\label{31:level}
J_{(1)}J=B\,\one,\qquad B=12k+\tfrac12(96-32)=4(3k+8),
\end{equation}
and, for a suitable normalization of $G^\pm$ \cite[(3.3)]{CL4} (see also
\cite{FS}, \cite{G}),
\begin{equation}\label{31:nu}
G^+_{(3)}G^-=\nu(k)\,\one,\qquad \nu(k)=(k+2)(2k+5)(3k+8).
\end{equation}
Since $G^\pm$ are quasiprimary and $J_{(0)}$ acts trivially on
$W^k(sl_4,f)_0$, the vectors $G^\pm$ lie in the maximal ideal of $W^k(sl_4,f)$
if and only if $\nu(k)=0$.

\begin{proposition}\label{prop:31}
Let $f\in sl_4$ be of Jordan type $(3,1)$. Then $W_k(sl_4,f)$ is unitary if and
only if $k\in\{-\frac83,-\frac52,-\frac73\}$. Moreover
$$
W_{-\frac83}(sl_4,f)=\C,\qquad W_{-\frac52}(sl_4,f)\cong \mathcal H,\qquad
W_{-\frac73}(sl_4,f)\cong V_{\Z\a}\ \text{ with }(\a|\a)=4 ,
$$
of central charge $0$, $1$, $1$ respectively.
\end{proposition}
\begin{proof}
We first show that $W_k(sl_4,f)$ is unitary at the three levels. The level
$k=-\frac83=-4+\frac43$ is admissible and $W_{-8/3}(sl_4,f)=\C$ by
\cite[Thm.~8.6]{AvEM}. Let $k=-\frac52$ or $k=-\frac73$. Then $B=2$, resp.\ $B=4$,
and $c=1$, which is the central charge of the Sugawara vector
$L^J=\frac1{2B}:JJ:$ of the Heisenberg subalgebra generated by $J$. The
hypotheses of \cite[Thm.~3.1]{AMP23} hold, because the strong generators of
conformal weight $2$ different from $L$ are $G^\pm$, and $(L-L^J)_{(2)}G^\pm$ is
a vector of conformal weight $1$, i.e.\ a multiple of $J$, on which $J_{(0)}$ acts by $\pm4$, hence it
vanishes. Therefore $L=L^J$ in $W_k(sl_4,f)$. By \cite[Thm.~3.2]{AMP23} a strong
generator of conformal weight $\Delta$ on which $J_{(0)}$ acts by $q$ can be
non-zero in $W_k(sl_4,f)$ only if $\frac{q^2}{2B}=\Delta$. This excludes $W$, and
it excludes $G^\pm$ for $k=-\frac52$, where $\frac{16}{2B}=4$. Hence
$W_{-5/2}(sl_4,f)\cong \mathcal H$ (see also \cite[Prop.~4.4]{AACLMMP}), which is
unitary. For
$k=-\frac73$ we have $\frac{16}{2B}=2$, and $G^\pm\ne0$ in $W_{-7/3}(sl_4,f)$
because $\nu(-\frac73)\ne0$. Thus $W_{-7/3}(sl_4,f)$ is a simple vertex algebra
with conformal vector $L^J$, strongly generated by $J$ and by the fields
$G^\pm$. By \cite[Prop.~2.6(i)]{AACLMMP} it is either $\mathcal H$ or a lattice vertex
algebra of rank one; since $G^\pm\ne0$, $J_{(0)}G^\pm=\pm4G^\pm$ and
$G^\pm$ have conformal weight $2=\frac{16}{2B}$, we obtain
$W_{-7/3}(sl_4,f)\cong V_{\Z\a}$ with $(\a|\a)=\frac{16}B=4$ (this is the
compactified free boson of \cite{ACGHR}), which is unitary \cite[Section 5.4]{KMP}.

Conversely, let $W_k(sl_4,f)$ be unitary with respect to a conjugate linear
involution $\phi$, and assume that $k\ne-\frac83$. Then $B\ne0$, so that $J\ne0$
in $W_k(sl_4,f)$. Write $\phi(J)=\zeta J$ with $|\zeta|=1$. Applying $\phi$ to
\eqref{31:level} we get $\zeta^2B=\ov B$, and by \eqref{eq:top}
\begin{equation}\label{31:JJ}
(J,J)=-\zeta B>0 .
\end{equation}
We claim that $k$ is real. If $G^+=0$ in $W_k(sl_4,f)$, then $\nu(k)=0$. If
$G^+\ne0$, then $\phi(G^+)$ is a non-zero vector such that
$J_{(0)}\phi(G^+)=\zeta^{-1}\phi(J_{(0)}G^+)=4\zeta^{-1}\phi(G^+)$, hence
$4\zeta^{-1}\in4\Z$, i.e.\ $\zeta=\pm1$, and $B$ is real.

Assume that $k<-\frac83$. Then $\nu(k)\ne0$, so that $G^+\ne0$  in $W_k(sl_4,f)$ and $\zeta=\pm1$;
since $B<0$, \eqref{31:JJ} gives $\zeta=1$. Then $J_{(0)}\phi(G^+)=4\phi(G^+)$,
so that $\phi(G^+)\in\C G^+$, and by \eqref{eq:top}
$(G^+,G^+)\one=\phi(G^+)_{(3)}G^+\in\C\,G^+_{(3)}G^+=0$, because
$G^+_{(3)}G^+$ is a vector of conformal weight $0$ on which $J_{(0)}$ acts by
$8$. This contradicts the positivity of the form.

Assume now that $k>-\frac83$, so that $B>0$. The vector $L^J=\frac1{2B}:JJ:$ is
fixed by $\phi$, since $\zeta^2/\ov B=1/B$. As in the proof of Lemma
\ref{lem:coset}, $J$ being primary of conformal weight $1$, the vector
$L'=L-L^J$ is quasiprimary, it is a Virasoro vector of central charge $c-1$
commuting with $J$, and $(L',L')=\frac12(c-1)$. Hence (for $t=k+h^\vee=k+4$)
$$
c-1=-\frac{4(2t-3)(3t-5)}t\ \ge0,\qquad\text{i.e.}\qquad \tfrac32\le t\le\tfrac53 .
$$
On this interval $c\le77-24\sqrt{10}<\frac32$. On the other hand, by
\eqref{eq:qpinv}, $(v,L'_nu)=(L'_{-n}v,u)$, so that $\one$ generates a unitary
highest weight module of the Virasoro algebra with highest weight $0$ and
central charge $c-1\in[0,\frac12)$. By \cite{FQS} this forces $c-1=0$, i.e.\
$t=\frac32$ or $t=\frac53$.
\end{proof}

\begin{remark}
The algebra $W^k(sl_4,f_{sub})$ is the algebra $W^{(2)}_4$ of \cite{FS}, see
\cite{G}. Its unitarity has been studied in \cite{ACGHR}, where the necessary condition
$\frac43\le k+4\le\frac{15}8$, i.e.\ $c\ge0$, is obtained from the norms of the
states of low conformal weight, and it is shown that at $k+4=\frac53$ the
algebra is a compactified free boson. Proposition \ref{prop:31} shows that
unitarity holds exactly at three points of this interval. The constraint which
is responsible for this is $c\ge c_{\g^\natural}=1$, sharpened by the
classification of the unitary representations of the Virasoro algebra with
central charge less than $1$.
\end{remark}

\section{Non-distinguished nilpotent elements}\label{sec:nd}

Let $\g$ be a simple Lie algebra and let $\mathfrak s=span\{e,x,f\}$ be an $sl_2$-subalgebra of
$\g$, with centralizer $\g^\natural$. Recall that the nilpotent element $f$, or the subalgebra
$\mathfrak s$, is called \emph{distinguished} if $\g^\natural=0$. Principal nilpotent elements
are distinguished, and in $sl_n$ there are no other distinguished nilpotent elements. In this
section we remove from Theorem \ref{prop:gen}(a) the assumption that $\mathfrak s$ is even and replace the assumption 
that $\g^\natural$ is not abelian by $\g^\natural\ne 0$.

\begin{theorem}\label{nd:main}
Let $\g$ be a simple Lie algebra and let $\mathfrak s$ be an $sl_2$-subalgebra of $\g$
which is not distinguished. Then $W_k(\g,\mathfrak s)$ can be unitary for at most a finite
number of $k$.
\end{theorem}

\begin{remark}\label{nd:rel}
(1) Principal $W$-algebras are unitary for infinitely many $k$ (Theorem \ref{12}). Hence, by
Theorem \ref{nd:main}, the finiteness of the set of unitary levels remains open only for the
distinguished non-principal $sl_2$-subalgebras, which exist exactly in types $B_n$ $(n\ge4)$, $C_n$ $(n\ge3)$, $D_n$ $(n\ge4)$, $E_6$, $E_7$, $E_8$, $F_4$, $G_2$.

(2) The proof of Theorem \ref{nd:main} distinguishes three cases: (a) $\mathfrak s$ is not
even; (b) $\mathfrak s$ is even and $\g^\natural$ is not abelian; (c) $\mathfrak s$ is even and
$\g^\natural$ is abelian and non-zero. Case (a) follows from Proposition \ref{previous} and
Lemma \ref{nd:finite}, and case (b) is Theorem \ref{prop:gen}(a); only case (c) requires a new
argument. In case (c), $\g^\natural$ is a torus, being an abelian Lie algebra of a reductive
group, the centralizer of $\mathfrak s$ in the adjoint group. Then
$\mathfrak l=\{a\in\g:[a,\g^\natural]=0\}$ is a Levi subalgebra of $\g$ containing
$\mathfrak s$, its centre is $\g^\natural$, and $f$ is distinguished in $\mathfrak l$. Thus
$\mathfrak l$ is the Levi subalgebra appearing in the Bala--Carter label of $f$, and
$\g^\natural$ is the centre of $\mathfrak l$.

(3) For classical $\g$, by Proposition \ref{prop:blocks},
case (c) occurs exactly when all the parts of the partition of $f$ have the same parity and,
moreover: for $sl_n$, the parts are distinct and their number $m$ is at least $2$, and then
$\g^\natural\cong\C^{m-1}$; for $so_N$ (all parts odd) and $sp_{2n}$ (all parts even), every
part has multiplicity at most $2$ and the number $r$ of parts of multiplicity $2$ is at least
$1$, and then $\g^\natural\cong so_2^{\oplus r}\cong\C^{r}$. For instance $\g^\natural\cong\C$
for the partition $(3,1)$ in $sl_4$, for which the unitary levels are determined in
Proposition \ref{prop:31}, and for the partition $(2,2)$ in $sp_4$.

(4) For exceptional $\g$, case (c) occurs exactly for the five orbits in the following table;
there is none in $G_2$ and $F_4$. We use the Bourbaki numbering of the simple roots
$\a_1,\dots,\a_\ell$. In each case $\g^\natural$ is spanned by the fundamental coweights
$\omega_i^\vee$ of the simple roots $\a_i$ not in $\mathfrak l$, so that $\g^\natural$ acts on a
root vector by the coefficients of these $\a_i$ in its root. The last column lists the non-zero
charges of $\g^\natural$ on $\g^f$ and, for each of them, the conformal weights of the
corresponding strong generators $J^{\{a\}}$ of $W^k(\g,\mathfrak s)$, one for each listed
weight.
\begin{center}
\renewcommand{\arraystretch}{1.2}
\small
\begin{tabular}{llll}
\hline
orbit & $\mathfrak l$ & $\g^\natural$ & charges: conformal weights of the charged generators\\
\hline
$E_6$, $D_4(a_1)$ & $D_4$ & $\C\omega_1^\vee\oplus\C\omega_6^\vee$ & $\pm(1,0)$, $\pm(0,1)$, $\pm(1,1)$: $2,3$\\
$E_6$, $D_5$ & $D_5$ & $\C\omega_6^\vee$ & $\pm1$: $3,6$\\
$E_7$, $E_6(a_1)$ & $E_6$ & $\C\omega_7^\vee$ & $\pm1$: $3,5,7$\\
$E_8$, $D_5+A_2$ & $D_5+A_2$ & $\C\omega_6^\vee$ & $\pm1$: $2,3,4,5,6,7$;\ $\pm2$: $2,4,5,6$;\ $\pm3$: $3,6$;\ $\pm4$: $2$\\
$E_8$, $D_7(a_1)$ & $D_7$ & $\C\omega_1^\vee$ & $\pm1$: $3,4,5,6,8,9$;\ $\pm2$: $2,6$\\
\hline
\end{tabular}
\end{center}
\end{remark}

We need two lemmas. The first one concerns the universal $W$-algebra. Recall (see the proof of
Proposition \ref{prop:conformal}) that $W^k(\g,\mathfrak s)$ is strongly generated by fields
$J^{\{a_i\}}$, where $\{a_i\}$ is a basis of $\g^f$ consisting of eigenvectors of $\ad x$, and
$J^{\{a_i\}}$ has conformal weight $\Delta_i=1-j$ if $[x,a_i]=ja_i$. More precisely, the
ordered monomials $:\!\partial^{m_1}J^{\{a_{i_1}\}}\cdots\partial^{m_s}J^{\{a_{i_s}\}}\!:$
form a basis of $W^k(\g,\mathfrak s)$ \cite[Theorem 4.1(b), Remark 4.2]{KW04}; we call them
\emph{PBW monomials}. There is a finite set $P\subset\C$, containing $-h^\vee$, such that for
$k\notin P$ the generators can be chosen in such a way that the structure constants of all
$n$-th products with respect to the PBW monomials are rational functions of $k$ without poles
outside $P$. Indeed, the complex of \cite[Section 4]{KW04} whose homology is
$W^k(\g,\mathfrak s)$ is a fixed graded vector space on which the $n$-th products and the
differential depend polynomially on $k$, and \cite[Theorem 4.1]{KW04} holds for every $k$
(cf.\ also \cite[\S5]{G} and the proof of Proposition \ref{prop:FFall}). Thus, for
$k\notin P$, the spaces $W^k(\g,\mathfrak s)$ are identified with a fixed graded vector space
with a fixed basis.

\begin{lemma}\label{nd:finite}
Let $v$ be a PBW monomial. Then the set $N(v)$ of the $k\in\C\setminus P$ such that $v$
belongs to the maximal ideal of $W^k(\g,\mathfrak s)$ is finite.
\end{lemma}
\begin{proof}
The ideal generated by $v$ is spanned by the vectors $u_{(n)}v$, where $u$ runs over the PBW
monomials and $n\in\Z$. It is proper if and only if it does not contain $\one$, i.e.\ if and
only if $\langle u_{(n)}v\rangle=0$ for all such $u$ and $n$. These expectation values are
rational functions of $k$ without poles outside $P$. If $N(v)$ were infinite, all of them would
vanish identically, hence $v$ would lie in the maximal ideal for every $k\notin P$. This is
not possible, since $v\ne0$ and $W^k(\g,\mathfrak s)$ is simple for all $k$ outside a countable set \cite[Thm.~3.5]{CLtr} (and the corollary following it).
\end{proof}

The second lemma refines Lemma \ref{lem:coset} in the case of a Heisenberg subalgebra.

\begin{lemma}\label{nd:bound}
Let $V$ be a simple unitary conformal vertex algebra, with conjugate linear involution $\phi$,
and let $J\in V_1$ be a non-zero even vector such that $\phi(J)=-J$. Then
$J_{(1)}J=(J,J)\,\one$ with $(J,J)>0$, and if $v\in V_\Delta$ is a non-zero
vector such that $J_{(0)}v=Qv$, then $Q\in\R$ and
\[
\Delta\ \ge\ \frac{Q^2}{2(J,J)}.
\]
\end{lemma}
\begin{proof}
By Theorem \ref{thm:kmp1}(1), $L_1J=0$, so that $J$ is primary and $g(J)=-\phi(J)=J$. Hence
\eqref{eq:top} gives $J_{(1)}J=(J,J)\one$, and \eqref{eq:qpinv} gives
$(v,J_nu)=(J_{-n}v,u)$, where $J_n=J_{(n)}$; in particular $Q$ is real. As in the proof of
Lemma \ref{lem:coset}, the vector $L'=L-L^J$, where $L^J=\frac1{2(J,J)}:\!JJ\!:$, is a
quasiprimary Virasoro vector commuting with $J$, and it is fixed by $\phi$. Therefore the
operators $L'_m$ commute with the operators $J_n$, satisfy the Virasoro relations and
$[L_0,L'_m]=-mL'_m$, and $(v,L'_mu)=(L'_{-m}v,u)$ by \eqref{eq:qpinv}.

We claim that $(w,L'_0w)\ge0$ for all $w\in V$. Since $L'_0$ is Hermitian and commutes with
$L_0$, we may assume that $L'_0w=\la w$ and $L_0w=dw$. The vectors
$L'_{m_1}\cdots L'_{m_r}w$ with all $m_i>0$ vanish if $\sum m_i>d$; choose a non-zero one,
$w'$, with $\sum m_i$ maximal (possibly $w'=w$). Then $L'_mw'=0$ for $m>0$ and
$L'_0w'=\la'w'$ with $\la'=\la-\sum m_i\le\la$, so that
$0\le(L'_{-1}w',L'_{-1}w')=(w',L'_1L'_{-1}w')=2\la'(w',w')$ and $\la\ge\la'\ge0$.

Finally, $L^J_0=\frac1{2(J,J)}J_0^2+\frac{1}{(J,J)}\sum_{j>0}J_{-j}J_j$, hence
\[
\Delta\,(v,v)=(v,L^J_0v)+(v,L'_0v)\ \ge\ \frac{Q^2}{2(J,J)}(v,v)+\frac{1}{(J,J)}\sum_{j>0}(J_jv,J_jv)\ \ge\ \frac{Q^2}{2(J,J)}(v,v).\qedhere
\]
\end{proof}

\begin{proof}[Proof of Theorem \ref{nd:main}]
Let $U$ be the set of the $k\ne-h^\vee$ such that $W_k(\g,\mathfrak s)$ is unitary, and let
$E=\bigcup_iN(J^{\{a_i\}})$ be the set of the $k\notin P$ at which some strong generator
vanishes in $W_k(\g,\mathfrak s)$; it is finite by Lemma \ref{nd:finite}.

If $\mathfrak s$ is not even, some strong generator $J^{\{a_i\}}$ is not compatible (Remark
\ref{lem:even}), hence it vanishes in $W_k(\g,\mathfrak s)$ for $k\in U$ by Proposition
\ref{previous}; thus $U\subseteq P\cup E$. If $\mathfrak s$ is even and $\g^\natural$ is not
abelian, $U$ is finite by Theorem \ref{prop:gen}(a). We may therefore assume that $\mathfrak s$ is
even and that $\g^\natural$ is abelian and non-zero.

Then $\g^\natural$ is the Lie algebra of a torus $T$ of the adjoint group, the centralizer of
$\mathfrak s$ in the adjoint group being reductive, and $\g^\natural$ is spanned by the
differentials of the cocharacters of $T$. Hence we may choose the basis $\{a_i\}$ of $\g^f$ in
such a way that, moreover, $[a,a_i]=\a_i(a)a_i$ for $a\in\g^\natural$, with
$\a_i\in(\g^\natural)^*$, and we can fix a non-zero $h\in\g^\natural$, the differential of a
cocharacter of $T$, so that $\ad h$ is diagonalizable on $\g$ with integer eigenvalues; in
particular $\a_i(h)\in\Z$ for all $i$. We choose the generators so that
${J^{\{a\}}}_{(0)}J^{\{a_i\}}=\a_i(a)J^{\{a_i\}}$ for $a\in\g^\natural$ \cite[Thm.~4.1]{AMP23};
then $J^{\{a\}}_{(0)}$ acts on a PBW monomial by the sum of the scalars $\a_i(a)$ attached to
its factors. An element of $\g^\natural$ which centralizes $\g^f$ is zero, because $\g$ is
generated by $\g^f$ as an $\mathfrak s$-module and has trivial centre. In particular
$q:=\a_{i_0}(h)\ne0$ for some $i_0$; set $G=J^{\{a_{i_0}\}}$ and $\Delta=\Delta_{i_0}$.
Normalize $(\cdot|\cdot)$ by \eqref{s:norm}, so that $\operatorname{tr}_\g(\ad h)^2=2h^\vee(h|h)$.
By the formula for the $\l$-bracket given in the proof of Theorem \ref{prop:gen}, where now
$\g_{1/2}=0$,
\begin{equation}\label{nd:Bk}
{J^{\{h\}}}_{(1)}J^{\{h\}}=B(k)\,\one,\qquad
B(k)=(k+h^\vee)(h|h)-\tfrac12\operatorname{tr}_{\g_0}(\ad h)^2 .
\end{equation}
Here $(h|h)>0$, because $2h^\vee(h|h)=\operatorname{tr}_\g(\ad h)^2$ is a sum of squares of
integers, which is non-zero since $h\ne0$ and $\g$ has trivial centre; similarly
$\operatorname{tr}_{\g_0}(\ad h)^2\ge0$.

Next we fix a constant $K$ and a vector $v_M$ which do not depend on $k$. By \eqref{c2}, in
which now $\g_{1/2}=0$, we have $c(\g,\mathfrak s,k)=Q(k)/(k+h^\vee)$, where $Q$ is a polynomial
of degree $2$ with leading coefficient $-12(x|x)<0$, depending only on $\mathfrak s$. If $Q$ has
a real root, let $k_+$ be the largest one and put $K=\max\{1,(k_++h^\vee)(h|h)\}$; otherwise
put $K=1$. Fix an integer $M>2K\Delta/q^2$ and let $v_M=\,:\!G^M\!:$. Recall that for $k\notin P$
the spaces $W^k(\g,\mathfrak s)$ are identified with a fixed graded vector space with a fixed
basis of PBW monomials; thus $v_M$ is one and the same PBW monomial for all $k\notin P$, and
$K$, $M$, $v_M$ depend only on $\mathfrak s$, $h$ and $i_0$. Moreover $v_M$ has conformal weight
$M\Delta$, and $J^{\{h\}}_{(0)}v_M=Mq\,v_M$.

Let $k\in U\setminus(P\cup E)$, let $\phi$ be the corresponding conjugate linear involution of
$W=W_k(\g,\mathfrak s)$, and denote by $J$ the image of $J^{\{h\}}$ in $W$. We claim that
$\phi(J)=-J$. Write $J=J'+\ii J''$ with $\phi(J')=J'$, $\phi(J'')=J''$. By \eqref{eq:qpinv},
$X=J'_{(0)}$ is skew-Hermitian and $Y=(\ii J'')_{(0)}$ is Hermitian, and they commute, because
$W_1$, being the image of $\g^\natural$, is an abelian Lie algebra. Hence on each $W_n$ they
are simultaneously diagonalizable, with purely imaginary, resp.\ real, eigenvalues. Since
$X+Y=J_{(0)}$ has integer eigenvalues, $X=0$. Now $J'$ is the image of $J^{\{y\}}$ for some
$y\in\g^\natural$, so that $X$ acts on the image of $J^{\{a_i\}}$, which is non-zero because
$k\notin E$, by $\a_i(y)$. Therefore $\a_i(y)=0$ for all $i$, i.e.\ $y$ centralizes $\g^f$;
thus $y=0$ and $J=\ii J''$, as claimed. Moreover $J\ne0$, since $J_{(0)}$ acts by $q\ne0$ on
the image of $G$.

By Lemma \ref{nd:bound} and \eqref{nd:Bk} we have $B(k)=(J,J)>0$; since $(h|h)>0$ and
$\operatorname{tr}_{\g_0}(\ad h)^2\ge0$, this implies that $k$ is real and $k>-h^\vee$.
Moreover $c(\g,\mathfrak s,k)=2(L,L)\ge0$ (see the proof of Theorem \ref{prop:gen}), hence
$Q(k)\ge0$. As $Q$ is concave, it has a real root and $k\le k_+$, so that
\begin{equation}\label{nd:K}
0<B(k)\le(k+h^\vee)(h|h)\le K .
\end{equation}
If the image of $v_M$ in $W$ were non-zero, Lemma \ref{nd:bound} and \eqref{nd:K} would give
\[
M\Delta\ \ge\ \frac{M^2q^2}{2B(k)}\ \ge\ \frac{M^2q^2}{2K},
\]
against the choice of $M$. Hence $v_M$ lies in the maximal ideal of $W^k(\g,\mathfrak s)$, i.e.\
$k\in N(v_M)$. Since $v_M$ does not depend on $k$, we have proved that
$U\subseteq P\cup E\cup N(v_M)$, which is a finite set by Lemma \ref{nd:finite}.
\end{proof}

By Theorem \ref{nd:main} and Remark \ref{nd:rel}, Conjecture \ref{nd:conj} below is open only
for the distinguished non-principal $sl_2$-subalgebras; in particular it holds in types $A_n$, $B_2$, $B_3$ (Remark \ref{nd:comments}(1)).

\begin{conjecture}\label{nd:conj}
Let $\g$ be a simple Lie algebra and $\mathfrak s$ a non-zero $sl_2$-subalgebra of $\g$. Then
$W_k(\g,\mathfrak s)$ is unitary for infinitely many $k$ if and only if $\mathfrak s$ is a
principal subalgebra.
\end{conjecture}

\begingroup
\begin{remark}\label{nd:comments}
We add some comments on Conjecture \ref{nd:conj}.

(1) The ``if'' part of Conjecture \ref{nd:conj} is Theorem \ref{12}, and the ``only if'' part
holds for $\mathfrak s$ not distinguished by Theorem \ref{nd:main}. Since in types $A_n$, $B_2$
and $B_3$ the principal nilpotent elements are the only distinguished ones (Remark
\ref{nd:rel}(1)), Conjecture \ref{nd:conj} holds for $\g$ of these types.

(2) If $\mathfrak s$ is distinguished, then $\g^\natural=0$, hence $W^k(\g,\mathfrak s)_1=0$, and
the methods of this section, which rely on the space of weight one, give no information. Some
information comes from the Virasoro subalgebra. Since $\mathfrak s$ is even, by \cite[Remark
2.2]{KRW} we have $c(\g,\mathfrak s,k)=A-\frac B{k+h^\vee}-C(k+h^\vee)$, with
$A=\dim\g_0+24(\rho|x)$, $B=12\|\rho\|^2$ and $C=12\|x\|^2$. If $k>-h^\vee$ is real and
$W_k(\g,\mathfrak s)$ is unitary, then $0\le c(\g,\mathfrak s,k)\le c_{\max}:=A-2\sqrt{BC}$
and, by the theorem of Friedan, Qiu and Shenker \cite{FQS} applied to the Virasoro module
generated by $\one$, either $c(\g,\mathfrak s,k)\ge1$ or $c(\g,\mathfrak s,k)=1-\frac6{m(m+1)}$
for some integer $m\ge2$. Since each value of $c$ is attained at most twice, there are only
finitely many such levels if $c_{\max}<1$. This happens for several distinguished nilpotent
elements, but not for all of them.

(3) Conjecture \ref{nd:conj} does not extend to basic Lie superalgebras, as minimal
$W$-algebras show \cite{KMP1}.
\end{remark}
\endgroup
\bigskip

{\bf Acknowledgements}. We would like to thank Claude
for helping us throughout the whole paper, especially  in the proof of Theorem \ref{thm:main}. We also thank  Dra\v{z}en Adamovi\'c,    Edward Frenkel and Andrew Linshaw for
correspondence.

\vskip5pt
    \footnotesize{
\noindent{\bf V.K.}: Department of Mathematics, MIT, 77
Mass. Ave, Cambridge, MA 02139;\newline
{\tt kac@math.mit.edu}
\vskip5pt
\noindent{\bf P.M-F.}: Politecnico di Milano, Polo regionale di Como,
Via Anzani 42, 22100, Como, Italy;\newline {\tt pierluigi.moseneder@polimi.it}
\vskip5pt
\noindent{\bf P.P.}: Dipartimento di Matematica, Sapienza Universit\`a di Roma, P.le A. Moro 2,
00185, Roma, Italy;\newline {\tt papi@mat.uniroma1.it}, Corresponding author
}
\end{document}